\documentclass[11pt,a4paper]{article}
	\usepackage[left=2.45cm, top=2.45cm,bottom=2.45cm,right=2.45cm]{geometry}
	\usepackage{mathtools,amssymb,amsthm,mathrsfs,calc,graphicx,xcolor,cleveref,dsfont,tikz,pgfplots,bm, extarrows}
	\usepackage{comment}
	\usepackage[british]{babel}
	\usepackage{amsfonts}              
	\usepackage{bbm}
	\usepackage[T1]{fontenc}
	\numberwithin{equation}{section}

	\newtheorem {theorem}{Theorem}[section]
	\newtheorem {proposition}[theorem]{Proposition}
	\newtheorem {lemma}[theorem]{Lemma}
	\newtheorem {corollary}[theorem]{Corollary}

	{\theoremstyle{definition}
		
	}
	{\theoremstyle{theorem}
		\newtheorem {remark}[theorem]{Remark}
		
	}

	\newcommand{\var}{\operatorname{Var}}

	\DeclareMathOperator{\arccosh}{arcosh}
	\DeclareMathOperator{\cum}{cum}

	\def\EE{\mathbb{E}}

	\def\HH{\mathbb{H}}

	\def\NN{\mathbb{N}}

	\def\RR{\mathbb{R}}
	\def\SS{\mathbb{S}}

	\def\XX{\mathbb{X}}

	\def\sfB{{\sf B}}

	\def\bM{\mathbf{M}}

	\def\cF{\mathcal{F}}
	
	\def\cH{\mathcal{H}}

	\def\dint{\textup{d}}
	\def\intd{\textup{d}}

	\DeclareMathOperator{\GOp}{\mathbf{G}}
	\DeclareMathOperator{\AOp}{\mathbf{A}}
	\DeclareMathOperator{\sn}{\mathbf{sn}}
	\DeclareMathOperator{\Gammo}{\mathbf{A}}
	\DeclareMathOperator{\indi}{\mathbbm{1}}
	
	\newcommand{\R}{\mathbb{R}}

	\newcommand{\xuta}{\xi}
	\newcommand{\mre}{e}
	\newcommand{\mri}{\mathrm{i}}
	
	\makeatletter
	\let\@fnsymbol\@alph
	\makeatother

\begin{document}
		
		\title{\bfseries {Intersections of Poisson $ k $-flats\\ in constant curvature spaces}}
		
		\author{Carina Betken\footnotemark[1]\;, Daniel Hug\footnotemark[2]\; and Christoph Th\"ale\footnotemark[3]}
		
		\date{}
		\renewcommand{\thefootnote}{\fnsymbol{footnote}}
		\footnotetext[1]{Ruhr University Bochum, Germany. Email: carina.betken@rub.de}
		\footnotetext[2]{Karlsruhe Institute of Technology (KIT), Germany. Email: daniel.hug@kit.edu}
		\footnotetext[3]{Ruhr University Bochum, Germany. Email: christoph.thaele@rub.de}
		
		\maketitle
		
		\begin{abstract}
			\noindent
			Poisson processes in the space of $k$-dimensional totally geodesic subspaces ($k$-flats) in
			a $d$-dimensional standard space of constant curvature $\kappa\in\{-1,0,1\}$ are studied, whose distributions are invariant under the isometries of the space. We consider the intersection processes of order $m$ together with their $(d-m(d-k))$-dimensional Hausdorff measure within a geodesic ball of radius $r$. Asymptotic normality for fixed $r$ is shown as the intensity of the underlying Poisson process tends to infinity for all $m$ satisfying $d-m(d-k)\geq 0$. For $\kappa\in\{-1,0\}$ the problem is also approached in the set-up where the intensity is fixed and $r$ tends to infinity. Again, if $2k\le d+1$ a central limit theorem is shown for all possible values of $m$. However, while for $\kappa=0$ asymptotic normality still holds if $2k>d+1$, we prove for $\kappa=-1$ convergence to a non-Gaussian infinitely divisible limit distribution in the special case $m=1$. The proof of asymptotic normality is based on the analysis of  variances and general bounds available from the Malliavin--Stein method. We also show for general $\kappa\in\{-1,0,1\}$ that, roughly speaking, the variances within a general observation window $W$ are maximal if and only if $W$ is a geodesic ball having the same volume as $W$. Along the way we derive  a new integral-geometric formula of Blaschke--Petkantschin type in a standard space of constant curvature.
		\end{abstract}
		

		\section{Introduction and statement of the results}
		
		Stochastic geometry deals with the development and the probabilistic and geometric analysis of models for complex spatial random structures, typically in a Euclidean space $\mathbb{R}^d$ of dimension $d\geq 2$. However, in recent years also stochastic geometry in non-Euclidean and especially in spherical and hyperbolic spaces has become an active field of research. The aim of this branch of stochastic geometry is to distinguish those properties of a random geometric system which are universal to some extent from the ones which are sensitive to the underlying geometry, especially to the curvature of the underlying space. We mention by way of example the studies \cite{BesauRosenThaele,BesauThaele,GodlandKabluchkoThaele} on random convex hulls, the papers \cite{BenjaminiVoronoi,HansenMueller0,HansenMueller1,HeroldHugThaele,HugThaeleSTIT,Isokawa,KabluchkoThaeleVoronoi,KabluchkoRosenThaele} on random tessellations as well as the works \cite{BenjaminiDelaunay,BodeEtAl,Flammant,FountoulakisVanDenHornEtAl,FountoulakisMuller,FountoulakisYukich,OwadaYogesh} on geometric random graphs and networks. The present paper continues this line of research and naturally connects to the articles \cite{HeroldHugThaele,KabluchkoRosenThaele}. We shall now explain our framework as well as our results.
		
		In this paper we deal with a $d$-dimensional standard space $\bM_\kappa^d$ of constant curvature $\kappa\in\{-1,0,1\}$. For $k\in\{0,1,\ldots,d-1\}$ we denote by $\AOp_\kappa(d,k)$ the space of $k$-flats, that is, the space of $k$-dimensional totally geodesic submanifolds, of $\bM_\kappa^d$. Each of the spaces $\AOp_\kappa(d,k)$ carries a  suitably normalized isometry invariant measure $\mu_{k,\kappa}$; the reader may consult Section \ref{sec:genset} for a detailed description. Next, for $t>0$ we let $\eta_{t,\kappa}$ be a Poisson process on $\AOp_\kappa(d,k)$ with intensity measure $t\mu_{k,\kappa}$ and refer to $\eta_{t,\kappa}$ as a \textbf{Poisson process of $k$-flats} in $\bM_\kappa^d$ of intensity $t$. To introduce the volume functional of intersection processes associated with $\eta_{t,\kappa}$, let $m\in\NN$ be such that $d-m(d-k)\geq 0$ and for a Borel set $W\subset\bM_\kappa^d$ define
		\begin{align}\label{eq:defFW}
			F_{W,t,\kappa}^{(m)} :&= \frac{1}{m!}\sum_{(E_1,\ldots,E_m)\in\eta_{t,\kappa,\neq}^m}\cH_\kappa^{d-m(d-k)}(E_1\cap\ldots\cap E_m\cap W) \nonumber\\
			&\qquad \qquad\qquad \qquad\qquad \times \mathbbm{1}\{ \text{dim}(E_1\cap\ldots\cap E_m)=d-m(d-k)\}.
		\end{align}
		Here, $\cH_\kappa^s$ for $s\geq 0$ denotes the $s$-dimensional Hausdorff measure with respect to the intrinsic metric $d_\kappa$ of $\bM_\kappa^d$, and we write $\eta_{t,\kappa,\neq}^m$ for the collection of all $m$-tuples of distinct $k$-flats in the support of $\eta_{t,\kappa}$. For example, $F_{W,t,\kappa}^{(1)}$ measures the total $\cH_\kappa^k$-volume in $W$ of the trace of all $k$-flats from $\eta_{t,\kappa}$, while if $m=\frac{d}{d-k}$ is an integer, then  $F_{W,t,\kappa}^{(m)}$ counts the number of points in $W$ that arise as intersection points of $m$-tuples of $k$-flats from $\eta_{t,\kappa}$. In classical stochastic geometry in Euclidean space, that is, for $\kappa=0$, central limit theorems for the centred and normalized versions of these random variables have been derived in \cite{Heinrich09,HugThaeleWeil,LastPenroseSchulteThaele} on different levels of generality. In fact, there are two basic set-ups for which one can study the fluctuations of $F_{W,t,\kappa}^{(m)}$:
		\begin{itemize}
			\item[(i)] For a fixed Borel set $W\subset\bM_\kappa^d$ with $\cH_\kappa^d(W)\in(0,\infty)$,  define
			\begin{equation}\label{eq:Fhat}
				\widehat{F}_{W,t,\kappa}^{(m)}:=\frac{F_{W,t,\kappa}^{(m)}-\EE F_{W,t,\kappa}^{(m)}}{\sqrt{\var F_{W,t,\kappa}^{(m)}}},
			\end{equation}
			and consider the asymptotics as $t\to\infty$.
			
			\item[(ii)] For $\kappa\in\{-1,0\}$ and for each $r\geq 1$, let $B_{r,\kappa}^d$ be a geodesic ball in $\bM_\kappa^d$, define
			\begin{equation}\label{eq:Ftilde}
				\widetilde{F}_{r,t,\kappa}^{(m)}:=\frac{F_{B_{r,\kappa}^d,t,\kappa}^{(m)}-\EE F_{B_{r,\kappa}^d,t,\kappa}^{(m)}}{\sqrt{\var F_{B_{r,\kappa}^d,t,\kappa}^{(m)}}},
			\end{equation}
			and for fixed $t>0$  consider the asymptotics as $r\to\infty$.
		\end{itemize}
		We start by considering the set-up described in (i). To measure the speed of convergence in the central limit theorem, we write ${\rm d}_W(X,Y)$ for the Wasserstein distance and ${\rm d}_K(X,Y)$ for the Kolmogorov distance between two random variables $X$ and $Y$, which are given by
		$$
		{\rm d}_\diamondsuit(X,Y) := \sup_{h\in\cF_\diamondsuit}|\EE h(X)-\EE h(Y)|,\qquad\diamondsuit\in\{W,K\},
		$$
		where $\cF_W$ is the class of Lipschitz functions on $\RR$ with Lipschitz constant $\leq 1$ and $\cF_K$ is the class of indicator functions of intervals of the form $(-\infty,x]$, $x\in\RR$. The constants $C, C_1,C_2,\ldots$ in  the forthcoming  theorems depend on the dimension $d$ only (further dependence on $m,k\le d-1$, for instance, can be subsumed under the dependence on $d$).

		\begin{theorem}[Central limit theorem for large intensities]\label{thm:IntensityToInfinity}
			Let $\kappa\in\{-1,0,1\}$ and consider a Poisson process of  $k$-flats in $\bM_\kappa^d$ with $d\geq 2$ and $k\in\{0,1,\ldots,d-1\}$. Let $m\in\NN$ be such that $d-m(d-k)\geq 0$. Let $ N $ be a standard Gaussian random variable, $\diamondsuit\in\{K,W\}$ and let $W\subset\bM_\kappa^d$ be a Borel set with $\cH_\kappa^d(W)\in(0,\infty)$. Then there is a constant $C\in(0,\infty)$  such that
			$$
			{\rm d}_\diamondsuit(\widehat{F}_{W,t,\kappa}^{(m)},N)\leq C\,t^{-1/2}
			$$
			for all $t\geq 1$. In particular, $\widehat{F}_{W,t,\kappa}^{(m)}$ satisfies a central limit theorem, as $t\to\infty$.
		\end{theorem}
		
		In fact, Theorem \ref{thm:IntensityToInfinity} is a  direct consequence of the general quantitative central limit theorem for Poisson U-statistics in  \cite[Theorem 4.7]{ReitznerSchulte} and \cite[Theorem 4.2]{Schulte}, see also \cite[Example 4.12]{EichelsbacherThaele} and Section \ref{sec:3.2} for an argument.
		
		It is an important observation that in Euclidean space, that is, for $\kappa=0$,  and if we take for $W$ a ball of radius $r>0$, the set-up considered in Theorem \ref{thm:IntensityToInfinity} is -- up to rescaling -- equivalent to considering a fixed intensity $ t $ and letting $ r $ grow to infinity at an appropriate speed. However, the equivalence breaks down for $\kappa=-1$. In fact, it was already shown in \cite{HeroldHugThaele} for $d\ge 4$ and $k=d-1$ that in hyperbolic space  no central limit theorem holds, and an extension of this finding is stated as Theorem~\ref{thm:NoCLT} of the present paper. Since in the Euclidean case $\kappa=0$ we have for all $k\in\{1,\ldots,d-1\}$ and $m\in\NN$ with $d-m(d-k)\geq 0$ that
		$$
		{\rm d}_\diamondsuit(\widetilde{F}_{r,t,0}^{(m)},N) \leq C\,r^{-\frac{d-k}{2}}
		$$
		for $r\geq 1$, where $C\in(0,\infty)$ is a constant, depending only on $d$, $k$ and $t$, we can from now on restrict our attention to the case $\kappa=-1$ of hyperbolic space. In fact, by the compactness of the spherical space $\bM_1^d$, spherical caps are bounded, which is the reason why in set-up (ii) we have restricted ourselves to the two non-compact space forms corresponding to $\kappa\in\{-1,0\}$. For simplicity of notation, let us assume that $t=1$ in what follows. Moreover we write $\widetilde{F}_r^{(m)}$ for $\widetilde{F}_{r,1,-1}^{(m)}$, $\HH^d$ for $\bM_{-1}^d$, $\AOp_h(d,k)$ instead of $\AOp_{-1}(d,k)$, and $\mu_{k}$ for $\mu_{k,-1}$.  We are now in the position to formulate a quantitative central limit theorem for $\widetilde{F}_r^{(m)}$, as $r\to\infty$, for particular choices of the parameters $d$, $k$ and $m$.
		
		\begin{theorem}[Central limit theorem for large radii and $\kappa=-1$]\label{thm:CLTintro}
			Consider a Poisson process of  $k$-flats in $\HH^d$ with $d\geq 2$ and $k\in\{0,1,\ldots,d-1\}$. Let $ N $ be a standard Gaussian random variable and $\diamondsuit\in\{K,W\}$.  For  $m\in \{1,2,3\}$ let $\widetilde{F}_r^{(m)}$ be the random variable defined at \eqref{eq:Ftilde}. Then there exist constants $ C_1, C_2, C_3 \in (0,\infty)$ such that the following assertions are true for any $r\geq 1$.
			\begin{itemize}
				\item[{\rm (i)}] If $  2k<d$, then $m=1$ and
				\begin{align}\label{CLT<d/2}
					{\rm d}_\diamondsuit(\widetilde{F}_r^{(m)},N)\leq C_1\,
					\begin{cases}
						e^{-\frac{r}{2}(d-2k+1)} &:\text{ for } k \geq 1,\\
						e^{-\frac{r}{2}(d-1)} &:\text{ for } k = 0.\\
					\end{cases}
				\end{align}
				\item[{\rm (ii)}] If $  2k=d $, then $m\in\{1,2\}$ and
				\begin{align}\label{CLT=d/2}
					{\rm d}_\diamondsuit(\widetilde{F}_r^{(m)},N)\leq C_2\,
					\begin{cases}
						e^{-\frac{r}{2}} &:\text{ for } d\geq 4,\\
						r^{m-1}	e^{-\frac{r}{2}} &:\text{ for } d=2.
					\end{cases}
				\end{align}
				\item[{\rm (iii)}] If $  2k=d+1 $, then $m\in\{1,2,3\}$ and
				\begin{align}\label{CLT=(d+1)/2}
					{\rm d}_\diamondsuit(\widetilde{F}_r^{(m)},N)\leq C_3\,
					\begin{cases}
						r^{-1}&:\text{ for } m=1,\\
						r^{-\frac{1}{2}} &:\text{ for } m\in\{2,3\}.
					\end{cases}			
				\end{align}
			\end{itemize}
			In particular, under  each of the assumptions {\rm (i)}, {\rm (ii)} or {\rm (iii)} the random variables $\widetilde{F}_r^{(m)}$ satisfy a central limit theorem, as $r\to\infty$.
		\end{theorem}
		
		\begin{remark}\rm
			For $ 2k\leq d+1 $, the intersection order $m$ can be at most  $ 2 $ for all $ d\in \{2,4,5,\ldots\}$, since $d-m(d-k)\geq 0$. Only in the exceptional case $ d=3 $ we can have the intersection order $m=3$. Thus, dealing only with $m\in \{1,2,3\}$ in Theorem \ref{thm:CLTintro} covers all possible cases, provided that $ 2k\leq d+1 $.
		\end{remark}

		The probabilistic analysis of the fluctuations of $\widetilde{F}_{r}^{(m)}$  in the special case $k=d-1$ has been carried out in \cite{HeroldHugThaele,KabluchkoRosenThaele}. It has been shown there that in this case a central limit theorem for $\widetilde{F}_r^{(m)}$ holds for the space dimensions $d=2$ and $d=3$; Theorem \ref{thm:CLTintro} recovers this result, but our argument is partly based on the previous work \cite{HeroldHugThaele}. In addition, it has also been shown in \cite{HeroldHugThaele} that there is no asymptotic normality for $d\geq 4$ if $m=1$ or $d\geq 7$ for arbitrary admissible $m$. In the special case $m=1$ the infinitely divisible non-Gaussian limit distribution for dimensions $d\geq 4$ has been identified in \cite[Theorem 2.1]{KabluchkoRosenThaele}. The following conjecture appears now natural in the light of Theorem \ref{thm:CLTintro} and the results just described.
		
		\bigskip
		
		\noindent
		\textbf{Conjecture.} \textit{Consider a  Poisson process of $k$-flats in $\HH^d$, $d\ge 4$, with $k\in\{0,1,\ldots,d-1\}$. For $r>0$ and $m\in\NN$ such that  $d-m(d-k)\geq 0$, let $\widetilde{F}_r^{(m)}$ be the random variable defined at \eqref{eq:Ftilde}. If $2k>{d+1}$, then the family of random variables $\widetilde{F}_r^{(m)}$ does not satisfy a central limit theorem, as $r\to\infty$.}
		
		\bigskip
		
		While we are not able to fully verify this conjecture, even not  in the case $k=d-1$ as explained in \cite{HeroldHugThaele}, we have the following partial result for $m=1$ which strongly supports the conjecture. In the following, we write $ \xlongrightarrow{D}  $ to indicate convergence in distribution. For integers $\ell\geq 1$, we set $\omega_\ell:=2\pi^{\ell/2}/\Gamma(\ell/2)$ for the surface measure of the Euclidean unit sphere of dimension $\ell-1$. Similarly as before, we write ${F}_r^{(m)}$ for ${F}_{W,1,-1}^{(m)}$ with $W=B^d_{r,-1}$. 
		
		In the following theorem,  $ \zeta $ denotes an inhomogeneous Poisson process on $ [0, \infty) $ with intensity function given by $s \mapsto \omega_{d-k}\cosh^{k} s\,\sinh^{d-k-1} s$. 
		
		\begin{theorem}[Non-Gaussian fluctuations for $m=1$ and $\kappa=-1$]\label{thm:NoCLT}
			Consider a  Poisson process of $k$-flats in $\HH^d$, where $d\geq 4$ and $k\in\{3,\ldots,d-1\}$.  If $2k>{d+1}$, then
			\begin{align*}
				\frac{F_r^{(1)}- \EE F_r^{(1)}}{e^{r(k-1)}}\xlongrightarrow{D} \frac{\omega_k}{(k-1)2^{k-2}}\,Z \quad \text{ as } r \rightarrow \infty,
			\end{align*}
			where $Z$ is the infinitely divisible, centred random variable given by
			\begin{equation}\label{eq:LimitVariableZ}
				Z := \lim_{T\to\infty}\Big(\sum_{s\in\zeta\cap[0,T]}\cosh^{-(k-1)} s-\frac{\omega_{d-k}}{d-k}\sinh^{d-k} T\Big)
			\end{equation}
			and $\zeta$ is an inhomogeneous Poisson process on $ [0, \infty) $ with intensity function given above. 
		\end{theorem}

		\begin{remark}\rm
			\begin{itemize}
				\item[(i)] By Proposition~\ref{prop:Variance} below, the rescaling $ e^{r(k-1)} $ in the previous theorem is of the same order as $\sqrt{\var F_r^{(1)}}$ as  $r\rightarrow \infty$, up to a multiplicative constant.
				\item[(ii)] As in \cite[Remark 2.3]{KabluchkoRosenThaele} one shows by means of a martingale argument that the limit in \eqref{eq:LimitVariableZ} exists almost surely and in $L^2$. The fact that $Z$ is infinitely divisible follows from the L\'evi--Khinchin formula and the explicit representation \eqref{eq:LimitCharacteristicFunction} of the characteristic function of $Z$, which we establish in the course of the proof of Theorem \ref{thm:NoCLT}. The latter also shows that $Z$ has no Gaussian component.
				To explain the centering in \eqref{eq:LimitVariableZ}, we consider
				$$
				Y_T:=\int \indi\{s\in [0,T]\}\cosh^{-(k-1)}s\, \zeta(\dint s).
				$$
				Then 
				$$
				\EE Y_T=\omega_{d-k}\int_0^T\cosh s\, \sinh^{d-k
					-1}s\, \dint s=\frac{\omega_{d-k}}{d-k}\sinh^{d-k} T$$
				and 
				$$
				\var Y_T=\omega_{d-k}\int_0^T\cosh^{2-k}s\, \sinh^{d-k-1}s\, \dint s.
				$$
				If $2k>d+1$, then $\lim\limits_{T\to\infty}\var Y_T<\infty$ which justifies the martingale argument mentioned above. 
				
				The L\'evy measure of $Z$ is concentrated on $(0,1)$ and arises as the image measure of the Lebesgue measure on $(0,\infty)$ with density $s \mapsto \omega_{d-k}\cosh^{k} s\,\sinh^{d-k-1} s$ under the mapping $s\mapsto\cosh^{-(k-1)}s$. Its Lebesgue density equals
				$$
				\rho(y)= \frac{\omega_{d-k}}{{k-1}}y^{-\frac{d+k-2}{k-1}}\left(1-y^{\frac{2}{k-1}}\right)^{\frac{d-k}{2}-1},\qquad y\in (0,1).
				$$
				Clearly, $\rho$ has a singularity at $0$ and the Lebesgue integral of $\rho$  over $(0,1)$ is infinite.  Moreover, the integrability of the function $y\mapsto y^2\rho(y)$ on $(0,1)$ can be seen from
				$$
				2-\frac{d+k-2}{k-1}>-1 \qquad\text{if and only if}\qquad d+1<2k.$$
				These findings are consistent with the results  obtained in   \cite{KabluchkoRosenThaele} in the case  where $k=d-1$.  
			\end{itemize}	
		\end{remark}

		Bounds for the growth of the variances as functions of the radius of a geodesic ball play an important role in the proof of Theorem \ref{thm:CLTintro} and especially Theorem \ref{thm:NoCLT}, see Proposition \ref{prop:Variance} below. Since we have explicit and unified formulas for the variances of the functionals $F_{W,t,\kappa}^{(m)}$ in an arbitrary observation window $W\subset\bM_\kappa^d$ and for general $\kappa\in\{-1,0,1\}$, it is natural to ask in this generality for which shapes $W$ the variances are maximal. The answer is given by Theorem \ref{covmax}, which states that the variances are maximal if $W$ is a geodesic ball in $\bM_\kappa^d$ which has the same volume as $W$. It seems that this result is new even in the Euclidean case $\kappa=0$.

		\begin{theorem}[Variance inequality and maximal variances]\label{covmax}
			Consider a Poisson process of  $k$-flats in $\bM_\kappa^d$ with $\kappa\in\{-1,0,1\}$, $d\geq 2$, and $k\in\{1,\ldots,d-1\}$. Let $W\subset\bM^d_\kappa$ be a Borel set with  $\cH_\kappa^d(W)\in(0,\infty)$, let $t>0$, and let $m\in\NN$ be such that $d-m(d-k)\geq 0$. In addition, suppose that $W$ is contained in a spherical cap of radius $\pi/4$ if $\kappa=1$. 
			If $B_W\subset\bM_\kappa^d$ is a geodesic ball with $\cH_\kappa^d(W)=\cH_\kappa^d(B_W)$, then
			$$
			\var F_{W,t,\kappa}^{(m)}\le \var F_{B_W,t,\kappa}^{(m)}.
			$$
			Equality holds if and only if there is an isometry $\varphi$ of $\bM^d_\kappa$ such that $W=\varphi(B_W)$, up to sets of $\cH_\kappa^d$-measure zero.
		\end{theorem}
		
		\begin{remark}\rm
			\begin{itemize}
				\item[(i)] We remark that the lower bound for $\var F_{W,t,\kappa}^{(m)}$ in Euclidean space is zero. For $d=2$ this can be checked by taking in  \cite[Lemma 6.1]{Heinrich09} a rectangle with side lengths $a=n$ and $b=1/n$, and then letting $n\to\infty$. Similar examples are possible in higher dimensions as well.
				\item[(ii)] A corresponding inequality also holds for the covariances between $F_{W,t,\kappa}^{(m_1)}$ and $F_{W,t,\kappa}^{(m_2)}$, where $m_1,m_2\in\NN$ satisfy $d-m_i(d-k)\geq 0$ for $i\in\{1,2\}$.
				\item[(iii)] If $k=0$, then $m=1$ and $ \var F_{W,t,\kappa}^{(1)}=t\mathcal{H}^d_\kappa(W)$. For this reason Theorem \ref{covmax} only deals with the case $k\geq 1$.
				\item[(iv)] For $\kappa=0$ and the volumes (in the appropriate dimensions) of the intersection processes of a Poisson hyperplane process, Heinrich has asked for the shape of an observation window (of given volume) such that the asymptotic variance under homothetic scaling of the window is maximal (see \cite[Section 6]{Heinrich09}). Theorem \ref{covmax} and its proof answers this question in generalized form. Some related chord power integrals are discussed in \cite{Heinrich16}. 
			\end{itemize}	
		\end{remark}

		A crucial tool in the proof of Theorem \ref{covmax} is a general sharp Riesz rearrangement inequality from \cite{BS01} and the following integral-geometric transformation formula of Blaschke--Petkantschin type for constant curvature spaces, which is of independent interest and which we could not locate in the existing literature. To present it, we need the modified sine function $\sn_\kappa:[0,\infty)\to[0,\infty)$, for $\kappa\in\{-1,0,1\}$, which is defined as
		\begin{equation}\label{gensin}
			\sn_\kappa(r) := \begin{cases}
				\sin r, & \kappa = 1,\\
				r, & \kappa = 0,\\
				\sinh r, & \kappa=-1,
			\end{cases}
		\end{equation}
		for $r\geq 0$. Recall that $d_\kappa$ denotes the intrinsic metric of  $\bM^d_\kappa$. 
		
		\begin{theorem}[Blaschke--Petkantschin type formula]\label{IGMeta}
			Let $\kappa\in\{-1,0,1\}$, $d\geq 2$, and $k\in\{1,\ldots,d-1\}$. If $f:\bM^d_\kappa\times\bM^d_\kappa\to[0,\infty]$ is a measurable function, then 
			\begin{align}\label{IGrel}
				&\int_{\AOp_\kappa(d,k)}\int_{E}\int_{E}f(x,y)\sn_\kappa^{d-k} d_\kappa(x,y)\,\cH_\kappa^k(\dint x)\,\cH_\kappa^k(\dint y)\,\mu_{k,\kappa}(\dint E)\nonumber\\
				&\qquad = \frac{\omega_k}{\omega_d}\int_{\bM_\kappa^d}\int_{\bM_\kappa^d}f(x,y)\,\cH^d_\kappa(\dint x)\,\cH^d_\kappa(\dint y).
			\end{align}
		\end{theorem}
		
		The Euclidean case $\kappa=0$ of Theorem \ref{IGMeta} is known in more general form, see \cite[Lemma 5.5]{Gardner}, where priority is given to \cite{Ren}. The approach in \cite{Gardner}  (for which help by Eva Vedel Jensen is acknowledged, see also \cite{EVJ}) is different from the present argument, also in the Euclidean case. We derive the result for $k\ge 2$ from the special case $k=1$ by another basic integral-geometric relation. For the case $k=1$ we provide a completely new approach in hyperbolic space ($\kappa=-1$) which allows us to deduce the result from the Euclidean Blaschke--Petkantschin formula via a suitable model of hyperbolic space (compare \cite[Equation (18.2)]{Santalo} for an approach via differential forms in the special case $k=1$). The current argument has the advantage of working in the same way in all three space forms simultaneously.
		
		\section{Preliminaries and preparations}

		\subsection{The standard spaces of constant curvature}\label{sec:genset}
		
		In this paper, we work in a $d$-dimensional standard space of constant curvature $\bM^d_\kappa$ with  $\kappa\in \{-1,0,1\}$ and intrinsic metric $d_\kappa$. An arbitrarily fixed  reference point in $\bM^d_\kappa$ (the ``origin'') will be denoted by $p$.  As the canonical model space for $\bM^d_0$, we use the Euclidean space $\R^d$ with Euclidean scalar product $\bullet$ and norm $\|\cdot\|$, and we choose $p=o$. The Euclidean unit sphere $\SS^d$ in the Euclidean space $\R^{d+1}$ will be the model space for $\bM^d_1$, and often it is convenient to choose an orthogonal coordinate system $(o,e_1,\ldots,e_{n+1})$ of $\R^{n+1}$ such that $p=e_{n+1}$ (the ``north pole''). Instead of $\bM^d_{-1}$ we prefer to write $\HH^d$ if only the hyperbolic space is considered. The  Beltrami--Klein model (sometimes also called projective ball model), based on the open Euclidean unit ball  $\sfB^d$ in $\R^d$,  will be a useful model space for the hyperbolic space $\HH^d$. For this model space the choice $p=o$ is convenient. For more specific information on the Beltrami--Klein model, we refer to \cite[Chapter 6]{Ratcliffe2019}.
		
		Recall that for $k\in \{0,1,\ldots,d-1\}$ $\AOp_\kappa(d,k)$ denotes the space of $k$-dimensional totally geodesic submanifolds of $\bM^d_\kappa$, which we call $k$-geodesics or $k$-flats, for short. We write $\GOp_\kappa(d,k)$ for the space of those elements of $\AOp_\kappa(d,k)$ that pass through the previously fixed origin $p$ of $\bM^d_\kappa$.  In particular, we write $\AOp_h(d,k)$ for $\AOp_{-1}(d,k)$ and $\GOp_h(d,k)$ for $\GOp_{-1}(d,k)$ when we are working in the hyperbolic space only.   In the model space $\R^d$ of $\bM_0^d$ the $k$-flats are $k$-dimensional affine subspaces of $\R^d$. In the model space $\SS^d$ of $\bM_1^d$, the $k$-flats are $k$-dimensional great subspheres of $\SS^d$, which arise as intersections of the $d$-dimensional unit sphere $\SS^d\subset\RR^{d+1}$ with $(k+1)$-dimensional linear subspaces of $\R^{d+1}$, that is elements of $\GOp_0(d+1,k+1)$. In the Beltrami--Klein model for $\bM_{-1}^d$, the $k$-flats are the non-empty  intersections of $k$-dimensional affine subspaces of $\R^d$ with the $d$-dimensional open unit ball $\sfB^d$. 
		
		Since the isometry group $I(\bM^d_\kappa)$ of $\bM^d_\kappa$ is unimodular (for $\kappa=-1$, see \cite[Proposition C.4.11]{BP92} or \cite[Chapter X, Proposition 1.4]{He62} together with the fact that $I(\bM^d_\kappa)$  is semi-simple as a Lie group) and $\AOp_\kappa(d,k)$ is a homogeneous $I(\bM^d_\kappa)$-space, there exists an $I(\bM^d_\kappa)$-invariant measure on $\AOp_\kappa(d,k)$, which is unique up to a constant factor. We write $\mu_{k,\kappa}$ for this measure and use the abbreviation $\mu_k$ if $\kappa=-1$. The normalization of $\mu_{k,0}$ is chosen as in \cite{SW} and the normalization (and parametrization)  of $\mu_{k}$  will be as in \cite[Equation (6)]{HeroldHugThaele}. More precisely, if we denote by $d_h$ the intrinsic metric of $\HH^d$, then
		\begin{align}\label{eq:mu_k}
			\mu_{k} (B)=\int_{\GOp_{h}(d,d-k)} \int_L \cosh^k d_h(x,p) \, \indi\{H(L,x)\in B\}\ \cH^{d-k}(\dint x)\, \nu_{d-k,h}(\dint L)
		\end{align}
		for a Borel set $ B\subset \AOp_{h}(d,k)$, where $ H(L,x) $ denotes the $k$-flat  passing through $ x\in L $ that is orthogonal to $ L $ at $x$, $ \nu_{d-k,h} $ is the Borel probability measure on the space $ \GOp_h(d,d-k)$, which is invariant under all isometries that fix the origin $ p $, and $\cH^{d-k}=\cH_{-1}^{d-k}$ stands for the Hausdorff measure on $L$ induced by the hyperbolic distance (see also the discussion below). If $k=0$, then \eqref{eq:mu_k} specializes to $\mu_0=\cH^d$.   Since $\AOp_1(d,k)$ is a compact space, the measure $\mu_{k,1}$ is often normalized as a probability measure. Instead we choose the normalization so that $\mu_{k,1}(\AOp_1(d,k))=\omega_{d+1}/\omega_{k+1}$. These choices ensure that the Crofton formula \cite[Lemma 2]{HeroldHugThaele} (see also \cite{Brothers}) holds in all three space forms with the same constants, provided that the normalization of the Hausdorff measures is chosen in a natural way. Namely, Hausdorff measures $\mathcal{H}^s_\kappa$, for $s\ge 0$, are defined (in each case) with respect to the underlying Riemannian metric (or the intrinsic metric) $d_\kappa$ and for $s=d$ they yield the natural volume measure on $\bM_\kappa^d$. If a Hausdorff measure $\mathcal{H}^k_\kappa$ is applied on a $k$-flat $E$ (with the induced Riemannian metric), we do not indicate $E$ in our notation (in particular if $E$ is clear from the context), since we always have $\mathcal{H}^k_{\kappa}=\mathcal{H}^k_{E,\kappa}$, where $\mathcal{H}^k_{E,\kappa}$ denotes the respective Hausdorff measure within $E$.
		
		We are now prepared to present the Crofton formula for the standard spaces $\bM_\kappa^d$. For the notion of Hausdorff rectifiability we refer to  \cite[Lemma 9, Remark 8]{HeroldHugThaele} (and the literature cited there) and remark that, for example, all compact  (geodesically) convex sets having the appropriate dimension satisfy this property.
		
		\begin{lemma}[Crofton formula]\label{le:CroftonFormula}
			Let $ 0\leq i \leq k \leq d-1 $, and let $ W\subset \bM_\kappa^d $ with $\kappa\in\{-1,0,1\}$ be a Borel set which is Hausdorff $ (d+i-k) $-rectifiable. Then
			\begin{align}\label{eq:CroftonFormula}
				\int_{\AOp_\kappa(d,k)} \cH_\kappa^i(W\cap E) \,\mu_{k,\kappa}(\dint E)=\frac{\omega_{d+1}\omega_{i+1}}{\omega_{k+1}\omega_{d-k+i+1}} \cH_\kappa^{d+i-k}(W).
			\end{align}
		\end{lemma}

		\subsection{Representation as Poisson U-statistics}
		
		Let $ \eta $ be a Poisson process on a measurable space $\XX$ with a non-atomic intensity measure. We then call a functional $ F=F(\eta)$ a Poisson U-statistic of order $ m \in\NN$ if $ F $ can be represented as
		\begin{align*}
			F(\eta)= \frac{1}{m!} \sum_{(x_1,\ldots, x_m) \in \eta_{\neq}^{m}} f(x_1,\ldots, x_m),
		\end{align*}
		with some measurable function $f:\XX^m\to[0,\infty]$, which is symmetric in its arguments and where $ \eta_{\neq}^m $ denotes the set of all $ m $-tuples of distinct points in the support of $\eta$. We call $f$ a kernel function for $F$. Poisson U-statistics have a variety of applications in stochastic geometry and we refer to \cite{LastPenrose,LastPenroseSchulteThaele,ReitznerSchulte} for further background material.
		
		We now fix a Borel set  $W\subset \bM_\kappa^d$ and consider the Poisson U-statistic $ F_{W,t,\kappa}^{(m)}$  of order $m$ on the space $\AOp_\kappa(d,k)$ with kernel $ f_\kappa:\AOp_\kappa(d,k)^m \rightarrow [0,\infty]  $  given by
		$$
		f_\kappa(E_1,\ldots,E_m):=\cH_\kappa^{d-m(d-k)}(E_1\cap\ldots\cap E_m\cap W)\, \indi\{ \text{dim}(E_1\cap\ldots\cap E_m)=d-m(d-k)\},
		$$
		while the underlying Poisson process $\eta_{t,\kappa}$ on $\AOp_{\kappa}(d,k)$ has intensity measure $t\mu_{k,\kappa}$. 
		The Poisson U-statistic $ F_{W,t,\kappa}^{(m)}$ admits the Wiener--It\^o chaos decomposition
		\begin{align} \label{eq:ChaosExpansion}
			{F}_{W,t,\kappa}^{(m)} = \EE[{F}_{W,t,\kappa}^{(m)}] + I_1(f_{W,t,\kappa,1}^{(m)}) + \ldots + I_m(f_{W,t,\kappa,m}^{(m)})
		\end{align}
		with  functions $f_{W,t,\kappa,i}^{(m)}:\AOp_\kappa(d,k)^i\to\RR$, $i\in\{1,\ldots,m\}$, given by
		\begin{align}\label{eq:kernels}
			&f_{W,t,\kappa,i}^{(m)}(E_1,\ldots, E_i) \nonumber\\
			&\qquad=\binom{m }{ i} \frac{t^{m-i}}{m!}\int_{\AOp_\kappa(d,k)^{m-i}}\cH_\kappa^{d-m(d-k)}(E_1\cap\ldots \cap E_i \cap E_{i+1}\cap \ldots \cap E_m \cap W)\notag \\
			&\hspace{3cm}\times\indi\{ \text{dim}(E_1\cap\ldots\cap E_m)=d-m(d-k)\}\  \mu_{k,\kappa}^{m-i} (\dint(E_{i+1},\ldots, E_m)),
		\end{align}
		where $I_i(\,\cdot\,)$ stands for the Wiener--It\^o integral with respect to the compensated Poisson process $\eta_{t,\kappa}-\mu_{k,\kappa}$, see \cite[Chapter 12]{LastPenrose} for further details. Note that for $\kappa=-1$ and $t=1$  the choice $ W=B_{r,-1}^d=:B^d_{r}$ yields
		\begin{equation}\label{eq:FrDefinition}
			F_r^{(m)} = F_{B^d_{r},1,-1}^{(m)}\quad\text{and}\quad \widetilde{F}^{(m)}_r=
			\frac{F_r^{(m)}-\EE F_r^{(m)}}{\sqrt{\var F_r^{(m)}}},
		\end{equation}
		for the functionals $ F_r^{(m)} $ and $\widetilde{F}_r^{(m)}$ involved in Theorem \ref{thm:CLTintro}.
		
		The following lemma (roughly speaking) shows that \emph{generically} the indicator on the right-hand side of \eqref{eq:kernels} is one if $\kappa=-1$ and the intersection of the $ k $-flats $ E_1, \ldots,E_m $ is non-empty, while for $\kappa\in\{0,1\}$ this is always the case. For convenience, we assign to the empty set the dimension $-1$.

		\begin{lemma}\label{lem:IntersectionDimension}
			Let $\kappa\in\{-1,0,1\}$. Let $d\ge 2$, $k\in\{0,1,\ldots,d-1\}$ and $r\in \{1,\ldots,\lfloor \frac{d}{d-k} \rfloor \} $. Then, for $\mu_{k,\kappa}^r$-almost all $(E_1,\ldots,E_r)\in\AOp_\kappa(d,k)^r$,
			$$
			\operatorname{dim}(E_1 \cap \ldots \cap E_r) \begin{cases}
				\in \{-1, d-r(d-k)\},&\text{if }\kappa=-1,\\
				=d-r(d-k),&\text{if }\kappa\in\{0,1\}.
			\end{cases}
			$$
		\end{lemma}

		\begin{proof}
			If $k=0$, then $r=1$, $E_1$ is a point and the assertion holds with $\operatorname{dim}(E_1)=0$ for each $\kappa\in\{-1,0,1\}$. We can thus assume that $k\ge 1$. We first consider the case $\kappa=1$. With $E_i\in \AOp_1(d,k)$ we associate the linear subspace $U_i\in \GOp_0(d+1,k+1)$ spanned by $E_i$. Then, \cite [Lemma 13.2.1]{SW} (or \cite[Lemma 4.4.1]{Schneider}) and an induction argument yield that
			$$
			\operatorname{dim}(U_1\cap\ldots\cap U_r)=(k+1)r-(r-1)(d+1)=d-r(d-k)+1
			$$
			for almost all $(U_1,\ldots,U_r)\in \GOp_0(d+1,k+1)^r$ with respect to the $r$-fold product measure of the Haar measure on $\GOp_0
			(d+1,k+1)$. Hence the assertion follows from the fact that $E_1\cap\ldots\cap E_r=U_1\cap\ldots\cap U_r\cap\mathbb{S}^d$.
			
			\medskip 
			
			For $\kappa=-1$ and $k=d-1$, the assertion has been proven in \cite{HeroldHugThaele}. We now extend the argument to the remaining cases $k\in\{1,\ldots,d-2\}$. For each $j\in\{1,\ldots,r\}$ we obtain a ``random uniform'' $ k $-flat $ E_j $ as the intersection of $ d-k $ ``independent random uniform''  hyperplanes $ H^{(j)}_1, \ldots, H^{(j)}_{d-k}\in \AOp_h(d,d-1)$, that is, for $ j \in \{1, \ldots ,r\} $ there are hyperplanes $ H^{(j)}_{1},  \ldots, H^{(j)}_{d-k} $ (all ``independent'') such that
			\[
			E_1 \cap \ldots \cap E_r=  \bigcap_{j=1}^r (H^{(j)}_1\cap \ldots\cap  H^{(j)}_{d-k}).
			\]
			More explicitly, by an $r$-fold application of \cite[Lemma  4]{HeroldHugThaele} we obtain that
			\begin{align}\label{eq:IntDimension}
				&\int_{\AOp_h(d,k)^r}
				\indi\{ \operatorname{dim}(E_1 \cap \ldots \cap E_r) \notin \{-1, d-r(d-k)\}\} \,\mu_k^r(\dint (E_1 ,\ldots , E_r))\notag\\
				&=c(d,k)^{-r}\int_{\AOp_h(d,d-1)^{r(d-k)}}
				\indi \Big\{ \operatorname{dim}( \bigcap_{j=1}^r (H^{(j)}_1\cap \ldots\cap  H^{(j)}_{d-k}) )\notin \{-1, d-r(d-k)\}\Big\} \notag\\
				&\qquad \qquad\qquad \qquad\qquad \qquad\qquad \qquad\qquad \qquad \times\mu_{d-1}^{r(d-k)}(\dint (H_1^{(1)}, \ldots  , H_{d-k}^{(r)})),
			\end{align}
			where
			\[
			c(d,k)=\frac{\omega_{k+1}}{\omega_{d+1}}\Big(\frac{\omega_{d+1}}{\omega_{d}}\Big)^{d-k}.
			\]
			It follows from \cite[Lemma 3]{HeroldHugThaele} that the right-hand side of \eqref{eq:IntDimension} vanishes, which implies that the integrand on the left side must vanish as well, for $\mu_k^r$-almost all $  (E_1,\ldots,E_r)\in \AOp_h(d,k)^r $. This proves the claim.
			
			\medskip 
			
			Finally, for $\kappa=0$ the proof (first for $k=d-1$, but then for all $k\in\{1,\ldots,d-1\}$) essentially follows in the same way as in the case $\kappa=-1$. We only have to observe that if $\operatorname{dim}(E_1 \cap \ldots \cap E_r)\neq d-r(d-k)$, then by basic facts of linear algebra we have $\operatorname{dim}(E_1 \cap \ldots \cap E_r)> d-r(d-k)$, in particular $E_1 \cap \ldots \cap E_r\neq\varnothing$.
		\end{proof}

		The next lemma extends Lemma 4 in \cite{HeroldHugThaele}. We write $\AOp^\ast_{\kappa}(d,k)^r$ for the set of all $(E_1,\ldots,E_r)\in \AOp_{\kappa}(d,k)^r$ with $E_1\cap\ldots\cap E_r\neq\varnothing$ if $\kappa=-1$ and set $\AOp^\ast_{\kappa}(d,k)^r:=\AOp_{\kappa}(d,k)^r$ if $\kappa\in\{0,1\}$.
		
		\begin{lemma}\label{new:Lemma4}
			Let $\kappa\in\{-1,0,1\}$. Let $d\ge 2$, $k\in\{0,1,\ldots,d-1\}$ and $r\in \{1,\ldots,\lfloor \frac{d}{d-k} \rfloor \} $. If
			$f:\AOp^\ast_{\kappa}(d,k)^r\to\R$ is a nonnegative, measurable function, then
			\begin{align*}
				&\int_{\AOp^\ast_{\kappa}(d,k)^r}f(E_1\cap\ldots \cap E_r)\, \mu^r_{k,\kappa}(\dint(E_1,\ldots,E_r))\\
				&\qquad= \frac{\omega_{d-r(d-k)+1}}{\omega_{d+1}}\left(\frac{\omega_{d+1}}{\omega_{k+1}}\right)^r
				\int_{\AOp_\kappa(d,d-r(d-k))}f(E)\, \mu_{d-r(d-k),\kappa}(\dint E).
			\end{align*}
			
		\end{lemma}
		
		\begin{proof}
			In the proof, we can proceed as in the proof of Lemma 4 in \cite{HeroldHugThaele}. We first observe that both sides of the asserted equation define isometry invariant Haar measures. Then we apply Lemma \ref{le:CroftonFormula} $r+1$ times to see that the constant is chosen correctly.
			
			Alternatively, observing first that Lemma 4 in \cite{HeroldHugThaele} holds for $\kappa\in\{-1,0,1\}$, we can apply this lemma $r+1$ times to obtain the assertion. Proceeding in this way, the constant on the right side is
			$$
			\frac{c(d,d-r(d-k))}{c(d,k)^r},
			$$
			which equals the constant given in the statement of the lemma.
		\end{proof}

		A consequence of the representation for $F_{W,t,\kappa}^{(m)}$ as given in \eqref{eq:ChaosExpansion}, is the following exact expression for its variance in terms of the functions $ f_{W,t,\kappa,i}^{(m)} $. From \cite[Proposition 12.12]{LastPenrose}
		we obtain
		\begin{equation}\label{eq:Variance}
			\var F_{W,t,\kappa}^{(m)}= \sum_{i=1}^m i!\,t^{2m-i} A_{W,\kappa,i}^{(m)},
		\end{equation}
		where $A_{W,\kappa,i}^{(m)}$ for $i\in\{1,\ldots,m\}$ is given by
		\begin{equation*}
			\begin{split}
				A_{W,\kappa,i}^{(m)} :=& \binom{m }{ i}^2 \int_{\AOp_{\kappa}(d,k)^{i}}\Bigg(\frac{1}{m!} \int_{\AOp_{\kappa}(d,k)^{m-i}}  \ \cH^{d-m(d-k)}_{\kappa}(E_1 \cap \ldots \cap E_i \cap E_{i+1}\cap \ldots \cap E_m\cap W)\\
				& \times\indi\{ \text{dim}(E_1\cap\ldots\cap E_m)=d-m(d-k)\}\, 
				\mu_{k,\kappa}^{m-i}(\dint(E_{i+1},\ldots,E_m))\Bigg)^2\\
				&\times\mu_{k,\kappa}^i(\dint(E_1,\ldots,E_i)).
			\end{split}
		\end{equation*}
		Note that in this formula we can restrict the integration to $m$-tuples of flats with non-empty intersections. Then Lemma \ref{lem:IntersectionDimension} shows that the indicator function can be replaced by $1$. Moreover, $\AOp_\kappa(d,k)^{m-i}$ can be replaced by $\AOp_\kappa^\ast(d,k)^{m-i}$ and $\AOp_\kappa(d,k)^{i}$  by $\AOp_\kappa^\ast(d,k)^{i}$. Then, for $\mu_{k,\kappa}^i$-almost all $(E_1,\ldots,E_i)\in \AOp_\kappa^\ast(d,k)^{i}$, we have $\operatorname{dim}(E_1 \cap \ldots \cap E_i)=d-i(d-k)$. Applying first Lemma \ref{new:Lemma4} to the integrations with respect to $(E_{i+1},\ldots,E_m)$ and to $(E_{1},\ldots,E_i)$, and then Lemma \ref{le:CroftonFormula} with respect to $F=E_{i+1}\cap \ldots\cap E_m$ and the $(d-i(d-k))$-rectifiable set $E\cap W=E_1\cap \ldots\cap E_i\cap W$, we obtain
		\begin{align*}
			A_{W,\kappa,i}^{(m)}&=\binom{m }{ i}^2\frac{1}{(m!)^2}\frac{c(d,d-(m-i)(d-k))^2}{c(d,k)^{2(m-i)}}\frac{c(d,d-i(d-k))}{c(d,k)^i}\nonumber\\
			&\quad \times \int_{\AOp_\kappa(d,d-i(d-k))}\left(\int_{\AOp_\kappa(d,d-(m-i)(d-k))}\cH^{d-m(d-k)}_\kappa(E\cap W\cap F)\, 
			\mu_{d-(m-i)(d-k),\kappa}(\dint F)\right)^2\nonumber\\
			&\quad\times \mu_{d-i(d-k),\kappa}(\dint E)\nonumber\\
			&=\binom{m }{ i}^2\frac{1}{(m!)^2}\frac{c(d,d-(m-i)(d-k))^2c(d,d-i(d-k))}{c(d,k)^{2(m-i)}c(d,k)^i} 
			\left(\frac{\omega_{d+1}\omega_{d-m(d-k)+1}}{\omega_{d-(m-i)(d-k)+1}\omega_{d-i(d-k)+1}}\right)^2\nonumber\\
			&\quad \times \int_{\AOp_\kappa(d,d-i(d-k))}\cH_\kappa^{d-i(d-k)}(E\cap W)^2\, \mu_{d-i(d-k),\kappa}(\dint E)\nonumber.
		\end{align*}
		After simplification of the constants, we finally get
		\begin{align}
			A_{W,\kappa,i}^{(m)}=&  \nonumber\binom{m }{i} ^2 \frac{1}{(m!)^2}\Big(\frac{\omega_{d+1}}{\omega_{k+1}}\Big)^{2m-i}\frac{\omega_{d-m(d-k)+1}^2}{\omega_{d+1}\omega_{d-i(d-k)+1}}\nonumber\\
			&    \times \int_{\AOp_\kappa(d,d-i(d-k)) }\ \cH_\kappa^{d-i(d-k)}(E \cap W )^2 \ \mu_{d-i(d-k),\kappa}(\dint E)\label{eq:DefArim}
		\end{align}
		for $i\in\{1,\ldots,m\}$.
		
		The same reasoning also shows that \eqref{eq:kernels} simplifies to 
		\begin{equation}\label{eq:fmw}
			f^{(m)}_{W,t,\kappa,i}(E_1,\ldots,E_i)=\binom{m}{i}\frac{t^{m-i}}{m!}\left(\frac{\omega_{d+1}}{\omega_{k+1}}\right)^{m-i}\frac{\omega_{d-m(d-k)+1}}{\omega_{d-i(d-k)+1}}\mathcal{H}_\kappa^{d-i(d-k)}(E_1\cap\ldots\cap E_i\cap W)
		\end{equation}
		for $\mu_{k,\kappa}^i$-almost all $(E_1,\ldots,E_i)\in \AOp_\kappa(d,k)^i$. Thus we have 
		$$
		t^{2m-i}A_{W,\kappa,i}^{(m)}=\int_{\AOp_{\kappa}(d,k)^{i}} f^{(m)}_{W,t,\kappa,i}(E_1,\ldots,E_i)^2\, (t\mu_{k,\kappa})^{i}(\dint (E_1,\ldots,E_i)),
		$$
		which is consistent with \eqref{eq:ChaosExpansion} and the isometry property of the Wiener--It\^o integrals. 
		
		\subsection{Integral asymptotics in the hyperbolic case}
		
		In this paper we already have and below will further encounter integral expressions of the form
		\begin{equation}\label{eq:intproto}
			\int_{\AOp_h(d,k)} \cH^k(E\cap B_r^d)^\ell\, \mu_k(\dint E),
		\end{equation}
		where $B_r^d=B_{r,-1}^d$ is a $d$-dimensional hyperbolic ball of radius $r>0$, $k\in\{0,1,\ldots,d-1\}$ and $\ell\in\NN$. Note that in this section we restrict ourselves to the case of the hyperbolic space with $\kappa=-1$ and use the notation introduced in Section \ref{sec:genset}. In particular, we will have to deal with the asymptotics of integrals of the form \eqref{eq:intproto}, as $r\to\infty$. Such an asymptotic analysis was already carried out in  \cite[Lemma 16]{HeroldHugThaele} and we recall the result here for completeness and in order to keep this paper self-contained.
		
		\begin{lemma}\label{lem:g}
			Let $ r\geq 1 $ and $ k \in \{0,1,\ldots, d-1\} $. For any $ \ell \in \NN $ there exist constants $ c,C >0 $, depending only on $ k ,\ell$ and $ d $, such that
			\begin{align*}
				c\, g(k,\ell,d,r)\ &\leq \ \int_{\AOp_h(d,k)} \cH^k(E\cap B_r^d)^\ell\, \mu_k(\dint E)\ \leq\ C\, g(k,\ell,d,r)
				\intertext{ with }
				g(k,\ell,d,r )&:=
				\begin{cases}
					e^{r(d-1)}&: \, \ell(k-1)< d-1,\\
					re^{r(d-1)}&: \, \ell(k-1)= d-1,\\
					e^{\ell r(k-1)}&: \, \ell(k-1)> d-1.
				\end{cases}
			\end{align*}
		\end{lemma}
		
		The previous lemma allows us to obtain the following asymptotic result for the quantities
		\begin{equation}\label{eq:AriAbbreviation}
			A_{r,i}^{(m)}:= A_{B_r^d,-1,i}^{(m)}
		\end{equation} 
		for which we derived in \eqref{eq:DefArim} a simplified representation in terms of an integral of the form \eqref{eq:intproto} with $\ell=2$.
		
		\medskip 
		
		In order to improve the readability of the subsequent arguments and results, we introduce the following notations: 
		Let $\XX$ be a set. We write $ f(x)\lesssim g(x) $ ($f(x)\gtrsim g(x)$)  for functions $ f,g :\mathbb{X} \rightarrow \RR$  if there exists a  constant $ C>0 $ such that $ f(x) \leq C g(x) $ ($f(x)\geq Cg(x)$) for all $ x \in \mathbb{X} $ and $f(x)  \asymp g(x) $ if there exist constants $ c,C >0 $ such that $ c \ f(x) \leq g(x) \leq C f(x) $ for all $ x\in \XX $. For a family of functions depending on some parameter $ r $ we write $ f_r(x) \sim f(x) $ if $f_r(x)/f(x) \rightarrow 1$ as $r \rightarrow \infty$ for each $ x\in \XX$. If we work in a $d$-dimensional standard space of constant curvature $\bM^d_\kappa$, $\kappa\in\{-1,0,1\}$,  the occurring constants may depend on the dimension $d$ (any dependence on other parameters can be subsumed under the dependence on $d$).
		
		\medskip 
		
		The following result deals again with the hyperbolic case.
		
		\begin{lemma}\label{lem:OrderArim}
			For all $ r\geq 1 $, $ k \in \{0,1,\ldots, d-1 \}$ and $ i \in \{1,\ldots, m\} $, 
			\begin{align*}
				A_{r,i}^{(m)} \ &\asymp
				\begin{cases}
					e^{r(d-1)}&: \, 2i(d-k)> d-1,\\
					re^{r(d-1)}&: \, 2i(d-k)= d-1,\\
					e^{2r(d-i(d-k)-1)}&: \, 2i(d-k)<d-1.
				\end{cases}
			\end{align*}
			In particular, if $ 2k< d+1 $, then
			\begin{equation}\label{eq:Arim_k<}
				A_{r,i}^{(m)}\asymp  e^{r(d-1)}\quad \text{for } i\in\{1,\ldots,m\},
			\end{equation}
			and if $ 2k= d+1 $, then 
			\begin{equation}\label{eq:Arim_k=}
				A_{r,1}^{(m)}\asymp  re^{r(d-1)} \quad \text{ and } \quad A_{r,i}^{(m)}\asymp  e^{r(d-1)}\quad \text{for }i\in\{2,\ldots,m\}.
			\end{equation}
		\end{lemma}
		\begin{proof}
			Following Lemma~\ref{lem:g}  we have
			\begin{align*}
				\int_{\AOp_h(d,d-i(d-k)) }\ \cH^{d-i(d-k)}(E \cap B_r^d )^2 \ \mu_{d-i(d-k)}(\dint E) \asymp g(d-i(d-k),2,d,r),
			\end{align*}
			so that
			\begin{align*}
				&A_{r,i}^{(m)} \ \asymp \begin{cases}
					e^{r(d-1)}&: \, 2(d-i(d-k)-1)<d-1,\\
					re^{r(d-1)}&: \, 2(d-i(d-k)-1)=d-1,\\
					e^{2r(d-i(d-k)-1)}&: \, 2(d-i(d-k)-1)>d-1.
				\end{cases}
			\end{align*}
			Now, $2(d-i(d-k)-1)<d-1$ if and only if $2i(d-k)>d-1$. Consequently, the condition $ 2(d-i(d-k)-1)>d-1 $ is equivalent to  $ 2i(d-k)<d-1 $ and  $ 2(d-i(d-k)-1)=d-1 $ is equivalent to  $ 2i(d-k)=d-1$, respectively. From here it is easy to see that if $ 2k<d+1$, then  we have $d-1<2(d-k) $, and if $ 2k=d+1$, then  $ d-1=2(d-k)$, yielding \eqref{eq:Arim_k<} and \eqref{eq:Arim_k=}, respectively.
		\end{proof}
		
		\section{Proofs of Theorems \ref{thm:IntensityToInfinity} and  \ref{thm:CLTintro}}

		\subsection{Asymptotic variance}
		
		A crucial ingredient in the proof of Theorem \ref{thm:CLTintro} is an asymptotic analysis of the variance of the random variables $F_r^{(m)}$. It turns out that the precise growth of $\var(F_r^{(m)})$ depends on the dimension parameter $k$ relative to the space dimension $d$, as the following result shows.
		
		\begin{proposition}\label{prop:Variance}
			Let $F_r^{(m)}$ be the random variable defined at \eqref{eq:FrDefinition}, for an underlying Poisson process on $\AOp_h(d,k) $ with $k\in\{0,1,\ldots,d-1\}$ and intensity $1$.   Then, as $r\to\infty$,  
			\begin{align*}
				\var F_r^{(m)} \asymp
				\begin{cases}
					e^{r(d-1)}&: \, 2k< d+1,\\
					re^{r(d-1)}&: \,  2k= d+1,\\
					e^{2r(k-1)}&: \, 2k>d+1.
				\end{cases}
			\end{align*}
		\end{proposition}
		\begin{proof}
			Recalling \eqref{eq:Variance}, \eqref{eq:DefArim} and \eqref{eq:AriAbbreviation} we need to determine the order of $ A_{r,i}^{(m)} $ for $ i \in \{1,\ldots,m\} $. We have to distinguish three cases. For $2k<d+1 $ the claim directly follows from \eqref{eq:Arim_k<} and for $2k={d+1}$ from \eqref{eq:Arim_k=}, respectively. If $ 2k >d+1 $, then $ d-1>2(d-k)$ and it follows from Lemma \ref{lem:OrderArim} that the  term $ A_{r,1}^{(m)} $ is of the order $ e^{2r(k-1)} $. Moreover, since $2(k-1)>d-1$ the term $A_{r,i}^{(m)}$ is of lower order for $i\ge 2$. 
		\end{proof}
		
		\subsection{Proofs of Theorem \ref{thm:IntensityToInfinity} and Theorem \ref{thm:CLTintro}}\label{sec:3.2}
		
		\begin{figure}[t]
			\centering
			\includegraphics[width=0.4\columnwidth]{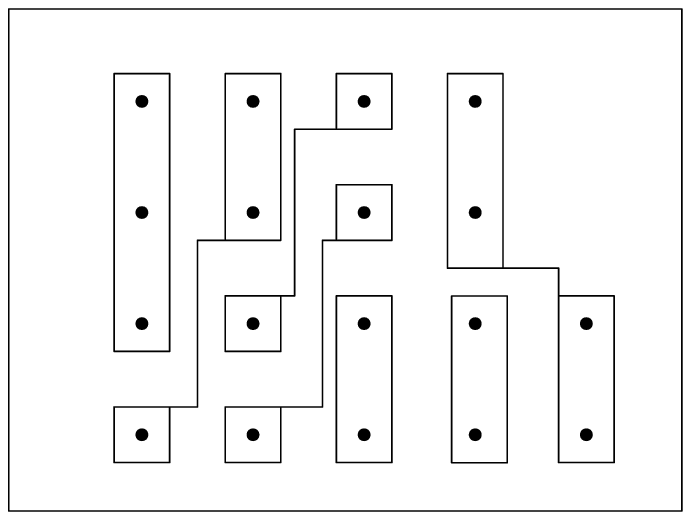}
			\caption{Example for a partition in $\Pi^{\rm con}_{\geq 2}(4,4,5,5)$.}
			\label{fig:Partition}
		\end{figure}
		
		To prove Theorems \ref{thm:IntensityToInfinity} and \ref{thm:CLTintro}, 
		we use the following general quantitative central limit theorem for Poisson U-statistics that can be found in \cite[Theorem 4.7]{ReitznerSchulte} and \cite[Theorem 4.2]{Schulte}.  Denoting by $ {\rm d}_\diamondsuit(\,\cdot\,,\,\cdot\,) $  either the Wasserstein ($\diamondsuit=W$) or the Kolmogorov ($\diamondsuit=K$) distance, this result states in our situation that there exists a constant $ c_{m,\diamondsuit} \in (0,\infty) $ (one can choose $ c_{m,W}=2m^{7/2} $ and $ c_{m,K}= 19m^5 $), such that
		\begin{align}\label{eq:CLTBound}
			{\rm d}_\diamondsuit\Bigg(\frac{F_{r,t,\kappa}^{(m)}-\EE F_{r,t,\kappa}^{(m)} }{\sqrt{\var F_{r,t,\kappa}^{(m)}}},N\Bigg)\leq c_{m,\diamondsuit} \sum_{u,v=1}^{m}
			\frac{\sqrt{M_{u,v}}}{\var F_{r,t,\kappa}^{(m)}},
		\end{align}
		where
		\begin{align} \label{eq:Muv}
			M_{u,v} =\sum_{\sigma \in \Pi^{\rm con}_{\geq 2}(u,u,v,v)} J(\sigma)
		\end{align}
		with
		\begin{align}\label{eq:Jsigma}
			J(\sigma):=\int_{\AOp_h(d,k)^{|\sigma|}}(f_{r,t,\kappa,u}^{(m)}\otimes f_{r,t,\kappa,u}^{(m)}\otimes f_{r,t,\kappa,v}^{(m)}\otimes f_{r,t,\kappa,v}^{(m)})_\sigma\,\dint(t\mu_k)^{|\sigma|},
		\end{align}
		and where $f_{r,t,\kappa,u}^{(m)}$ and $f_{r,t,\kappa,v}^{(m)}$ are defined by \eqref{eq:kernels}, but with $W=B^d_r$.  
		Let us explain the notation in \eqref{eq:Muv}. For  $n=2u+2v$ and a partition $\sigma$ of $[n]=\{1,\ldots,n\}$  into non-empty (disjoint) subsets (called blocks), $(f_{r,t,\kappa,u}^{(m)}\otimes f_{r,t,\kappa,u}^{(m)}\otimes f_{r,t,\kappa,v}^{(m)}\otimes f_{r,t,\kappa,v}^{(m)})_\sigma$ stands for the tensor product of these functions in which all variables belonging to the same block of $\sigma$ have been identified. Also, $|\sigma|$ denotes the number of blocks of $\sigma$. The set of partitions $\Pi^{\rm con}_{\geq 2}(u,u,v,v)$ in \eqref{eq:Muv} is defined as follows. We visualize the elements of $ [2u+2v]=\{1,\ldots,2u+2v\}$ by a diagram of points arranged in $4$ rows, where row $ i $ has precisely $ u $ elements for $ i \in\{1,2\} $ and precisely $ v $ elements for $ i \in \{3,4\} $, respectively, representing the arguments of the $i$-th function in the tensor product. The blocks of a partition $\sigma$ are indicated by closed curves, where the elements enclosed by a curve indicate that these elements belong to the same block of $\sigma$. That a partition $\sigma$ of $[2u+2v]$ belongs to the set $\Pi^{\rm con}_{\geq 2}(u,u,v,v)$ means that
		\begin{itemize}
			\item[(a)] all blocks of $\sigma$ have at least two elements;
			\item[(b)] each block of $\sigma$ contains at most one element from each row;
			\item[(c)] the diagram representing $\sigma$ is connected, meaning that the rows cannot be divided into two subsets each defining a separate diagram.
		\end{itemize}
		We refer to Figure \ref{fig:Partition} for an illustration and to \cite{HeroldHugThaele,LastPenroseSchulteThaele,PeccatiTaqquBook} for further background material on partitions.
		
		\medskip 
		
		We start by proving Theorem \ref{thm:IntensityToInfinity}.

		\begin{proof}[Proof of Theorem \ref{thm:IntensityToInfinity}]
			The relations \eqref{eq:CLTBound}, \eqref{eq:Muv}, and \eqref{eq:Jsigma} hold in fact for a general observation window $W\subset\bM_\kappa^d$ satisfying $\cH_\kappa^d(W)\in(0,\infty)$, for general $t>0$ and $\kappa\in \{-1,0,1\}$. It follows from \eqref{eq:Variance} that $\var(F_{W,t,\kappa}^{(m)})\gtrsim t^{2m-1}$. Moreover, in order to bound $M_{u,v}$ in \eqref{eq:Muv} for $u,v\in\{1,\ldots,m\}$ from above, we count in \eqref{eq:Jsigma} the powers of the intensity $t$. First, from \eqref{eq:fmw} it follows that each of the functions $f_{r,t,\kappa,u}^{(m)}$ and $f_{r,t,\kappa,v}^{(m)}$ contributes a power $t^{m-u}$ and $t^{m-v}$, respectively. Finally, the integration with respect to $(t\mu_k)^{|\sigma|}$ gives the factor $t^{|\sigma|}$, while all other terms are constants independent of $t$. Thus, $J(\sigma)\lesssim t^{2(m-u)+2(m-v)+|\sigma|}$ and
			$$
			\sqrt{M_{u,v}}\lesssim \max_{\sigma\in\Pi_{\geq 2}^{\rm con}(u,u,v,v)}\left(t^{2(m-u)+2(m-v)+|\sigma|}\right)^{\frac{1}{2}}\leq  \left(t^{2(m-u)+2(m-v)+2(u+v)-3}\right)^{\frac{1}{2}}= t^{2m-\frac{3}{2}},
			$$
			since $|\sigma|\le 2(u+v)-3$ for $\sigma\in \Pi^{\rm con}_{\geq 2}(u,u,v,v)$. Plugging this bound into \eqref{eq:CLTBound}, the proof of Theorem \ref{thm:IntensityToInfinity} is  completed.
		\end{proof}

		\medskip 
		
		We now turn to the proof of Theorem~\ref{thm:CLTintro}, which basically follows the same line of arguments as the proof of \cite[Theorem 5]{HeroldHugThaele}, but needs a suitable adaption to our situation. To simplify notation we set $t=1$ and omit the indices $t=1$ and $\kappa=-1$ in all expressions.
		
		\begin{proof}[Proof of Theorem \ref{thm:CLTintro}]
			In the following, we repeatedly use that 
			$$ 
			f^{(m)}_{r,i}(E_1,\ldots,E_i) \asymp \mathcal{H}^{d-i(d-k)}(E_1\cap\ldots\cap E_i\cap B^d_r).
			$$ 
			for $\mu_{k}^i$-almost all $(E_1,\ldots,E_i)\in \AOp_h(d,k)^i$.
			
			\paragraph{Case 1: $\mathbf{ m=1} $.}
			In order to bound the right-hand side of \eqref{eq:CLTBound} in this case, we only need to deal with
			\begin{align*}
				M_{1,1}=\sum_{\sigma \in \Pi^{\rm con}_{\geq 2}(1,1,1,1)} J(\sigma) \lesssim\int_{\AOp_h(d,k)} \cH^{k}(E\cap B_r^d)^4 \,\mu_k(\dint E),
			\end{align*}
			since there is only one possible partition in $\Pi^{\rm con}_{\geq 2}(1,1,1,1)$, as depicted in the left panel of Figure~\ref{fig:PartitionsM1}.
			\begin{figure}[t]
				\centering
				\includegraphics[width=0.9\columnwidth]{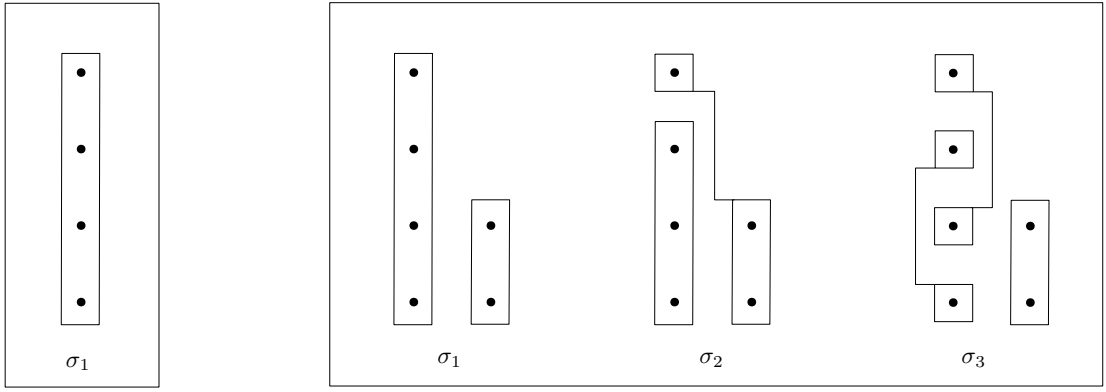}
				\caption{ Left: The only possible partition in $\Pi^{\rm con}_{\geq 2}(1,1,1,1)$. Right: The partitions in $\Pi^{\rm con}_{\geq 2}(1,1,2,2)$. }
				\label{fig:PartitionsM1}
			\end{figure}
			
			\subparagraph{Sub-case 1.1: $\mathbf{ k \in \{2,\ldots, \lfloor\frac{d}{2}\rfloor\}} $.}  According to \cite[Lemma 7]{HeroldHugThaele}, we then have
			\begin{align}\label{eq:kdimball}
				\cH^{k}(E\cap B_r^d) \lesssim  e^{r(k-1)}
			\end{align}
			for $\mu_{k}$-almost all $ E \in \AOp_h(d,k) $, so that
			\begin{align}\label{eq:BoundM11}
				M_{1,1}&\lesssim \, e^{2r(k-1)} \int_{\AOp_h(d,k)} \cH^{k}(E\cap B_r^d)^2 \mu_k(\dint E) \lesssim  \,  e^{2r(k-1)} \, g(k,2,d,r)
				\lesssim  e^{r(2k+d-3)}
			\end{align}
			by Lemma~\ref{lem:g}.
			
			\subparagraph{Sub-case 1.2: $\mathbf{k\in\{0,1\}}$.} In this situation, Lemma~\ref{lem:g} implies that
			\begin{align*}
				M_{1,1}& \lesssim \int_{\AOp_h(d,k)} \cH^{k}(E\cap B_r^d)^4\, \mu_k(\dint E) \lesssim \, g(k,4,d,r) \asymp \,  e^{r(d-1)}.
			\end{align*}
			Combining these results with the lower bound on $ \var(F_r^{(1)}) $ in Proposition~\ref{prop:Variance}, we get \eqref{CLT<d/2} as well as  \eqref{CLT=d/2}  in the case $ m=1 $.
			
			\subparagraph{Sub-case 1.3: $\mathbf{2k=d+1}$.} To establish \eqref{CLT=(d+1)/2} in this case, note that
			\begin{align*}
				M_{1,1}&\lesssim  \int_{\AOp_h(d,k)} \cH^{k}(E\cap B_r^d)^4\,\mu_k(\dint E) \lesssim \, g(k,4,d,r) \asymp \,  e^{4r(k-1)}= e^{2r(d-1)},
			\end{align*}
			which in combination with Proposition~\ref{prop:Variance} yields \eqref{CLT=(d+1)/2} for $ m=1 $ as well.
			
			\medskip
			
			\paragraph{Case 2: $ \mathbf{m=2 }$.} Since $d-2(d-k)= 2k-d < 0$ if $2k <d$, this case can only occur  if $ d=2k$ or $ d=2k-1$. In these cases, it remains to bound $ M_{1,2} $ and $ M_{2,2} $.
			
			\subparagraph{Sub-case 2.1: $ \mathbf{2k=d}$.}  The proof of \cite[Theorem 5 (a)]{HeroldHugThaele} shows that in order to bound $ M_{1,2} $ we only have to deal with the partitions depicted in the right panel of Figure~\ref{fig:PartitionsM1} (up to relabelling of the elements). Before we present our estimates, note that the case $ d=2 $ and $ k=1 $ was covered by \cite[Theorem 5(a)]{HeroldHugThaele}, so that we can assume $ d\geq 4 $ and thus $ k\geq 2 $.  First, for $ \sigma_1$ as shown on the right in Figure~\ref{fig:PartitionsM1} we have
			\begin{align*}
				J(\sigma_1)&\lesssim \int_{\AOp_h(d,k)^{2}} \cH^{k}(E_1\cap B_r^d)^2 \, \cH^{0}(E_1\cap E_2\cap B_r^d)^2\,\mu^2_k(\dint(E_1,E_2))\\
				&\lesssim e^{r(k-1)} \int_{\AOp_h(d,k)^{2}} \cH^{k}(E_1\cap B_r^d) \, \cH^{0}(E_1\cap E_2\cap B_r^d)\,\mu^2_k(\dint(E_1,E_2)),
			\end{align*}
			where we used \eqref{eq:kdimball} (recall that $k\ge 2$) and bounded $ \cH^{0}(E_1\cap E_2\cap B_r^d) $ by 1.
			An application of the Crofton formula \eqref{eq:CroftonFormula} for the integration with respect to $ E_2 $ shows that
			\begin{align*}
				J(\sigma_1)
				&\lesssim  e^{r(k-1)}\int_{\AOp_h(d,k)} \cH^{k}(E_1\cap B_r^d)^2 \,\mu_k(\dint E_1) \lesssim e^{r(k-1)} g(k,2,d,r) \lesssim e^{r(d+k-2)}
			\end{align*}
			by Lemma~\ref{lem:g}.
			Similarly, for $\sigma_2$ we have
			\begin{align*}
				J(\sigma_2)&\lesssim  \int_{\AOp_h(d,k)^{2}} \cH^{k}(E_1\cap B_r^d) \, \cH^{k}(E_2\cap B_r^d)\, \cH^{0}(E_1\cap E_2\cap B_r^d)^2\,\mu^2_k(\dint(E_1,E_2))\\
				&\lesssim e^{r(k-1)}\int_{\AOp_h(d,k)} \cH^{k}(E_1\cap B_r^d)^2\mu_k(\dint E_1) \lesssim e^{r(d+k-2)},
			\end{align*}
			and for $\sigma_3$ we obtain
			\begin{align*}
				J(\sigma_3)&\lesssim  \int_{\AOp_h(d,k)^{3}} \cH^{k}(E_1\cap B_r^d) \, \cH^{k}(E_2\cap B_r^d)\, \cH^{0}(E_1\cap E_3\cap B_r^d)\,\\
				&\hspace{4cm} \times \cH^{0}(E_2\cap E_3\cap B_r^d)\,\mu^3_k(\dint(E_1,E_2,E_3))\\
				&\lesssim e^{2r(k-1)}\int_{\AOp_h(d,k)} \cH^{k}(E_3\cap B_r^d)^2\, \mu_k(\dint E_3) \lesssim e^{r(2(k-1)+(d-1))}=e^{r(d+2k-3)},
			\end{align*}
			where we used  \eqref{eq:kdimball} twice, applied the Crofton formula \eqref{eq:CroftonFormula} for the integration with respect to $ E_1 $ and $ E_2 $  and bounded the last integral with the help of Lemma~\ref{lem:g}. Combining the bounds on $ J(\sigma_1) $, $ J(\sigma_2) $ and $ J(\sigma_3) $, we now obtain
			\begin{align}\label{eq:BoundM12}
				M_{1,2}\lesssim e^{r(d+k-2)}+e^{r(d+2k-3)} \asymp e^{r(d+2k-3)}.
			\end{align}
			To bound the remaining term $ M_{2,2} $, we first note that the partitions in $\Pi^{\rm con}_{\geq 2}(2,2,2,2)$ are (up to reordering of the elements in the diagram) given in Figure~\ref{fig:PartitionsM22} (see the proof of \cite[Theorem 5(a)]{HeroldHugThaele}).
			\begin{figure}[t]
				\centering
				\includegraphics[width=0.9\columnwidth]{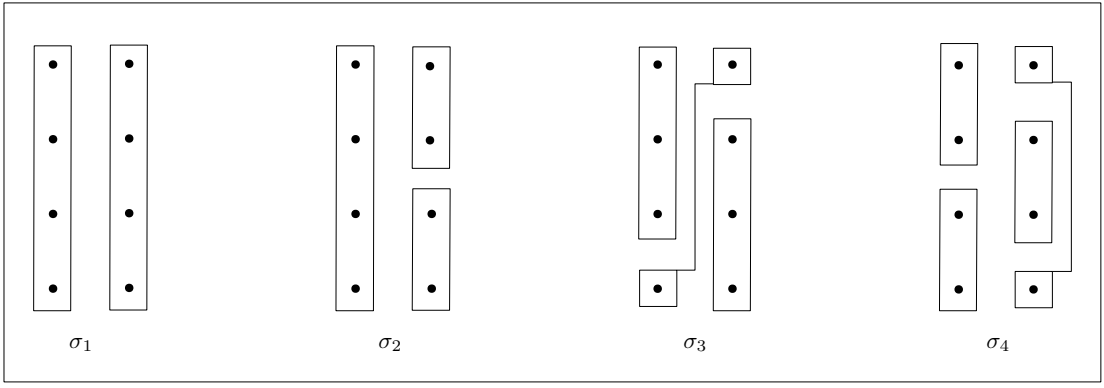}
				\caption{The four possible partitions in $\Pi^{\rm con}_{\geq 2}(2,2,2,2)$. }
				\label{fig:PartitionsM22}
			\end{figure}
			For $\sigma_1$ we have
			\begin{align*}
				J(\sigma_1) &\lesssim \int_{\AOp_h(d,k)^{2}} \cH^{0}(E_1\cap E_2 \cap B_r^d)^4\,\mu^2_k(\dint(E_1,E_2))\\
				&\leq \int_{\AOp_h(d,k)^2} \cH^{0}(E_1\cap E_2 \cap B_r^d)\,\mu^2_k(\dint(E_1,E_2)) \\
				&\lesssim \cH^{d}(B_r^d) \lesssim e^{r(d-1)},
			\end{align*}
			where we used the trivial bound  $ \cH^{0}(E_1\cap E_2 \cap B_r^d)\leq 1 $   and applied Crofton's formula \eqref{eq:CroftonFormula} afterwards. For $\sigma_2$ we compute
			\begin{align*}
				J(\sigma_2) &\lesssim \int_{\AOp_h(d,k)^{3}} \cH^{0}(E_1\cap E_2 \cap B_r^d)^2\,\cH^{0}(E_1\cap E_3 \cap B_r^d)^2\,\mu^3_k(\dint(E_1,E_2, E_3))\\
				&\leq \int_{\AOp_h(d,k)^{3}} \cH^{0}(E_1\cap E_2 \cap B_r^d)\,\cH^{0}(E_1\cap E_3 \cap B_r^d)\,\mu^3_k(\dint(E_1,E_2, E_3))\\
				&\lesssim \int_{\AOp_h(d,k)} \cH^{k}(E_1 \cap B_r^d)^2\,\mu_k(\dint E_1)
				\lesssim e^{r(d-1)},
			\end{align*}
			where we bounded  $ \cH^{0}(E_1\cap E_2 \cap B_r^d) $ and  $ \cH^{0}(E_1\cap E_3 \cap B_r^d) $  by $1$, applied the Crofton formula \eqref{eq:CroftonFormula} for the integration with respect to $ E_2 $ and $ E_3 $  and used Lemma~\ref{lem:g} afterwards. Similarly, for $\sigma_3$ we find that
			\begin{align*}
				J(\sigma_3) &\lesssim \int_{\AOp_h(d,k)^{3}} \cH^{0}(E_1\cap E_2 \cap B_r^d)\, \cH^{0}(E_1\cap E_3 \cap B_r^d)^2\,\cH^{0}(E_2\cap E_3 \cap B_r^d)\,\mu^3_k(\dint(E_1,E_2,E_3))\\
				&\leq \int_{\AOp_h(d,k)^{3}} \cH^{0}(E_1\cap E_2 \cap B_r^d)\,\cH^{0}(E_2\cap E_3 \cap B_r^d)\,\mu^3_k(\dint(E_1,E_2,E_3))\\
				&\lesssim \int_{\AOp_h(d,k)} \cH^{k}( E_2 \cap B_r^d)^2\,\mu_k(\dint E_2)\lesssim e^{r(d-1)},
			\end{align*}
			where we first bounded the quadratic term by $1$, applied the Crofton formula \eqref{eq:CroftonFormula} for the integration with respect to $ E_1 $ and $ E_3 $ and used Lemma~\ref{lem:g}. For the last partition $ \sigma_4 $ we have
			\begin{align*}
				J(\sigma_4) &\lesssim \int_{\AOp_h(d,k)^{4}} \cH^{0}(E_1\cap E_2 \cap B_r^d)\, \cH^{0}(E_1\cap E_3 \cap B_r^d)\,\cH^{0}(E_4\cap E_3 \cap B_r^d)\,\\
				&\qquad \qquad \qquad \qquad \qquad \qquad\times\cH^{0}(E_4\cap E_2 \cap B_r^d)\,\mu^4_k(\dint(E_1,E_2,E_3,E_4))\\
				&\leq \int_{\AOp_h(d,k)^{4}}\cH^{0}(E_1\cap E_3 \cap B_r^d)\,\cH^{0}(E_4\cap E_3 \cap B_r^d)\cH^{0}(E_4\cap E_2 \cap B_r^d)\,\mu^4_k(\dint(E_1,E_2,E_3,E_4))\\
				&\lesssim \int_{\AOp_h(d,k)^2} \cH^{k}( E_3 \cap B_r^d)\,\cH^{0}(E_3\cap E_4 \cap B_r^d)\,\cH^{k}( E_4 \cap B_r^d)\,\mu^2_k(\dint (E_3,E_4)),
			\end{align*}
			where, similarly to the preceding cases, we first bounded $ \cH^{0}(E_1\cap E_2 \cap B_r^d) $ by one and applied \eqref{eq:CroftonFormula} for the integration with respect to $ E_2 $ and $ E_1 $. Using \eqref{eq:kdimball} to bound $ \cH^k(E_3\cap B_r^d) $, Crofton's formula \eqref{eq:CroftonFormula} and Lemma~\ref{lem:g} again, we obtain
			\begin{align*}
				J(\sigma_4)&\lesssim e^{r(k-1)} \int_{\AOp_h(d,k)} \cH^{k}( E_4 \cap B_r^d)^2\,\mu_k(\dint E_4)\lesssim e^{r(d+k-2)}.
			\end{align*}
			As a consequence, we deduce the bound
			\begin{align}\label{eq:BoundM22}
				M_{2,2} \lesssim   e^{r(d-1)}+ e^{r(d+k-2)} \lesssim  e^{r(d+k-2)}.
			\end{align}
			Combination of the estimates \eqref{eq:BoundM11} for $ M_{1,1} $, \eqref{eq:BoundM12} for $ M_{1,2} $ and \eqref{eq:BoundM22} for $ M_{2,2} $ with Proposition~\ref{prop:Variance} yields
			\begin{align*}
				d\big(\tilde{F}_r^{(m)},N\big)\lesssim e^{-r(d-1)}\big(e^{\frac{r}{2}(2d-3)}+e^{\frac{r}{2}(2d-3)}+e^{\frac{r}{2}(\frac{3}{2}d-2)}\big) \lesssim e^{-\frac{r}{2}}
			\end{align*}
			in the case $ 2k=d $, which proves \eqref{CLT=d/2} for $ m=2 $.
			
			\subparagraph{Sub-case 2.2: $\mathbf{ 2k=d+1} $.} In this case we have $k\ge 2$.  We first consider $M_{1,2}$. Starting with the first group of partitions of the form $ \sigma_1 $ in the right panel of Figure~\ref{fig:PartitionsM1} we obtain
			\begin{align*}
				J(\sigma_1)&\lesssim  \int_{\AOp_h(d,k)^{2}} \cH^{k}(E_1\cap B_r^d)^2 \, \cH^{1}(E_1\cap E_2\cap B_r^d)^2\,\mu^2_k(\dint(E_1,E_2))\\
				&\lesssim r \int_{\AOp_h(d,k)} \cH^{k}(E_1\cap B_r^d)^3 \,\mu_k(\dint E_1) \lesssim r e^{3r(k-1)},
			\end{align*}
			where we  bounded $ \cH^{1}(E_1\cap E_2\cap B_r^d) $ by $ 2r $, used Crofton's formula \eqref{eq:CroftonFormula} and Lemma~\ref{lem:g}. Proceeding in a similar way for the two other partitions, we get
			\begin{align*}
				J(\sigma_2)&\lesssim  \int_{\AOp_h(d,k)^{2}} \cH^{k}(E_1\cap B_r^d) \, \cH^{k}(E_2\cap B_r^d)\, \cH^{1}(E_1\cap E_2\cap B_r^d)^2\,\mu^2_k(\dint(E_1,E_2))\\
				&\lesssim r e^{r(k-1)}\int_{\AOp_h(d,k)} \cH^{k}(E_1\cap B_r^d)^2\mu_k(\dint E_1) \asymp r^2 e^{r(d+k-2)}
			\end{align*}
			by \eqref{eq:kdimball}, the bound $ \cH^{1}(E_1\cap E_2\cap B_r^d)\leq 2r $, Crofton's formula \eqref{eq:CroftonFormula} and Lemma~\ref{lem:g}. Furthermore,
			\begin{align*}
				J(\sigma_3)&\lesssim  \int_{\AOp_h(d,k)^{3}} \cH^{k}(E_1\cap B_r^d) \, \cH^{k}(E_2\cap B_r^d)\, \cH^{1}(E_1\cap E_3\cap B_r^d)\,\\
				&\hspace{4cm} \times \cH^{1}(E_2\cap E_3\cap B_r^d)\,\mu^3_k(\dint(E_1,E_2,E_3))\\
				&\lesssim e^{2r(k-1)}\int_{\AOp_h(d,k)} \cH^{k}(E_3\cap B_r^d)^2\, \mu_k(\dint E_3) \asymp r e^{r(2k-2+d-1)}=r e^{2r(d-1)},
			\end{align*}
			where we used  \eqref{eq:kdimball}  twice, applied \eqref{eq:CroftonFormula} for the integration with respect to $ E_1 $ and $ E_2 $  and bounded the last integral with the help of Lemma~\ref{lem:g}. Combining the bounds on $ J(\sigma_1) $, $ J(\sigma_2) $ and $ J(\sigma_3) $, we arrive at
			\begin{align}\label{eq:BoundM12odd}
				M_{1,2}\lesssim r(e^{3r(k-1)}+re^{r(d+k-2)}+e^{r(2d-2)} )\asymp r e^{2r(d-1)}.
			\end{align}
			To derive an upper bound for $ M_{2,2} $, we again consider the partitions depicted in Figure~\ref{fig:PartitionsM22}. Before we start, we need to introduce some additional notation.  For $ E \in \AOp_h(d,k) $ we denote by $ L_1(E)\in\AOp_h(d,1)$ an arbitrary $1$-flat  which satisfies $L_1(E)\subset E$ and $ d_h(L_1(E),p)=d_h(E,p) $. 
			For $\sigma_1$ we now have
			\begin{align*}
				J(\sigma_1) &\lesssim \int_{\AOp_h(d,k)^{2}} \cH^{1}(E_1\cap E_2 \cap B_r^d)^4\,\mu^2_k(\dint(E_1,E_2))\\
				&\leq \int_{\AOp_h(d,k)^2} \cH^{1}(E_1\cap E_2 \cap B_r^d)\,\cH^{1}(L_1(E_1) \cap B_r^d)^3\,\mu^2_k(\dint(E_1,E_2)) \\
				&\lesssim \int_{\AOp_h(d,k)} \cH^{1}(L_1(E_1) \cap B_r^d)^3\,\cH^{k}(E_1\cap B_r^d)\,\mu_k(\dint E_1),
			\end{align*}
			where we used Crofton's formula \eqref{eq:CroftonFormula}. Moreover, 
			note that $L_1(E_1) \cap B_r^d $ is a 1-dimensional ball of radius  $ \arccosh(\cosh(r)/\cosh(d_h(E_1,p)) )$, and thus in particular
			\begin{align}\label{eq:1dimball}
				\cH^1(L_1(E_1) \cap B_r^d) \leq 2r
			\end{align}
			for $\mu_{1}$-almost all $ E_1 \in \AOp_h(d,1) $ which intersect $B^d_r$. Using \eqref{eq:1dimball} and  \eqref{eq:CroftonFormula}, we see that
			\begin{align*}
				J(\sigma_1) &\lesssim r^3 \,\cH^d(B_r^d)\asymp r^3 \, e^{r(d-1)}.
			\end{align*}
			With similar considerations for $\sigma_2$, we compute
			\begin{align*}
				J(\sigma_2) &\lesssim \int_{\AOp_h(d,k)^{3}} \cH^{1}(E_1\cap E_2 \cap B_r^d)^2\,\cH^{1}(E_1\cap E_3 \cap B_r^d)^2\,\mu^3_k(\dint(E_1,E_2, E_3))\\
				&\lesssim \int_{\AOp_h(d,k)^{3}} \cH^{1}(L_1(E_1)\cap B_r^d)^2\,\cH^{1}(E_1\cap E_2 \cap B_r^d)\,\cH^{1}(E_1\cap E_3 \cap B_r^d)\ \mu^3_k(\dint(E_1,E_2, E_3))\\
				&\lesssim r^2\int_{\AOp_h(d,k)}\cH^{k}(E_1 \cap B_r^d)^2\,\mu_k(\dint E_1)
				\asymp r^2 g(k,2,d,r) \asymp r^3 e^{r(d-1)},
			\end{align*}
			where we applied \eqref{eq:CroftonFormula} for the integration with respect to $ E_2 $ and $ E_3 $, bounded  $ \cH^{1}(L_1(E_1) \cap B_r^d)^2 $  by $4r^2$ and used Lemma~\ref{lem:g} afterwards. Similarly, for $\sigma_3$ we find that
			\begin{align*}
				J(\sigma_3) &\lesssim \int_{\AOp_h(d,k)^{3}} \cH^{1}(E_1\cap E_2 \cap B_r^d)\, \cH^{1}(E_1\cap E_3 \cap B_r^d)^2\,\cH^{1}(E_2\cap E_3 \cap B_r^d)\,\mu^3_k(\dint(E_1,E_2,E_3))\\
				&\leq  \int_{\AOp_h(d,k)^{3}} \cH^{1}(E_1\cap E_2 \cap B_r^d)\,\cH^{1}(E_2\cap E_3 \cap B_r^d)\,\cH^{1}(L_1(E_1) \cap B_r^d)^2\,\mu^3_k(\dint(E_1,E_2,E_3))\\
				&\lesssim r^2\int_{\AOp_h(d,k)} \cH^{k}( E_2 \cap B_r^d)^2\,\mu_k(\dint E_2)\lesssim r^3 e^{r(d-1)},
			\end{align*}
			where we first bounded the quadratic term by a  $4r^2$ according to \eqref{eq:1dimball}, applied \eqref{eq:CroftonFormula} for the integration with respect to $ E_1 $ and $ E_3 $ and used Lemma~\ref{lem:g}. For the last partition $ \sigma_4 $ we have
			\begin{align*}
				J(\sigma_4) &\lesssim \int_{\AOp_h(d,k)^{4}} \cH^{1}(E_1\cap E_2 \cap B_r^d)\, \cH^{1}(E_1\cap E_3 \cap B_r^d)\,\cH^{1}(E_4\cap E_3 \cap B_r^d)\,\\
				&\qquad \qquad \qquad \qquad \qquad \qquad\times\cH^{1}(E_4\cap E_2 \cap B_r^d)\,\mu^4_k(\dint(E_1,E_2,E_3,E_4))\\
				&\leq \int_{\AOp_h(d,k)^{4}}\cH^{1}(L_1(E_1)\cap B_r^d)\,\cH^{1}(E_1\cap E_3 \cap B_r^d)\,\cH^{1}(E_4\cap E_3 \cap B_r^d)\\
				&\qquad \qquad \qquad \qquad \qquad \qquad \times\cH^{1}(E_4\cap E_2 \cap B_r^d)\,\mu^4_k(\dint(E_1,E_2,E_3,E_4))\\
				&\lesssim r\int_{\AOp_h(d,k)^2} \cH^{k}( E_3 \cap B_r^d)\,\cH^{1}(E_3\cap E_4 \cap B_r^d)\,\cH^{k}( E_4 \cap B_r^d)\,\mu^2_k(\dint (E_3,E_4))\\
				&\lesssim r e^{r(k-1)}\int_{\AOp_h(d,k)} \cH^{k}( E_4 \cap B_r^d)^2\,\mu_k(\dint E_4) \asymp r^2 e^{r(k+d-2)},
			\end{align*}
			where, similarly to the preceding cases, we first used \eqref{eq:1dimball} to bound $ \cH^{1}(L_1(E_1) \cap B_r^d) $, applied \eqref{eq:CroftonFormula} for the integration with respect to $ E_2 $ and $ E_1 $ and then used \eqref{eq:kdimball} and Lemma~\ref{lem:g}.
			As a consequence, we deduce the bound
			\begin{align}\label{eq:BoundM22odd}
				M_{2,2} \lesssim  r^3 e^{r(d-1)}+ r^2 e^{r(d+k-2)} \lesssim  r^2 e^{r(d+k-2)}.
			\end{align}
			Combining \eqref{eq:BoundM11}, \eqref{eq:BoundM12odd} and \eqref{eq:BoundM22odd} with Proposition~\ref{prop:Variance}, we finally get
			\begin{align*}
				d\big(\tilde{F}_r^{(m)},N\big)\lesssim (re^{-r(d-1)})^{-1}\big(e^{r(d-1)}+r^{\frac{1}{2}}e^{r(d-1)}+r e^{\frac{r}{2}(\frac{3}{2}d-\frac{3}{2})}\big) \lesssim r^{-\frac{1}{2}}
			\end{align*}
			in the case $ 2k=d+1 $, which completes the proof of \eqref{CLT=(d+1)/2} also for $ m=2 $.
			
			\paragraph{Case $3$: $\mathbf{m=3}$}
			As noted before, the case $ m=3 $ can only occur in dimension $ d=3 $ and for $ k=2 $.  This situation has already been treated in  \cite[Theorem 5 (b)]{HeroldHugThaele}, where it was shown that $\tilde{F}_r^{(3)}$ converges in distribution towards a standard Gaussian random variable at rate $ r^{-\frac{1}{2}} $ for both the Wasserstein and the Kolmogorov distance. This finishes the proof of Theorem~\ref{thm:CLTintro}.
		\end{proof}

		\section{Proof of Theorem \ref{thm:NoCLT}}
		We fix $ k \in \{2,\ldots,d-1\} $ and proceed analogously to the proof of Theorem 2.1 in \cite{KabluchkoRosenThaele}. In the proof, we write $F_r$ instead of $F^{(1)}_r$ for short. To derive the result, we show that the characteristic function of the random variable  $ \frac{F_r - \EE F_r }{e^{r(k-1)}} $ converges towards the characteristic function of $\frac{\omega_k}{(k-1)2^{k-2}}Z $, as $ r \rightarrow \infty $, where $ Z $ is the random variable defined in the statement of the theorem. 
		
		We start with some preparations. For $s\ge 0$ choose any $L_0\in\GOp_h(d,d-k) $ and $x_0\in L_0$ with $d_h(x_0,p)=s$. Then we define $f_r(s):=\mathcal{H}^k(H(L_0,x_0)\cap B^d_r)$, where $H(L_0,x_0)$ was introduced at \eqref{eq:mu_k}. The definition of the function $f_r:[0,\infty)\to [0,\infty)$ is  independent of the particular choices of $L_0$ and $x_0$. Since $p$ is fixed,  we shortly write $d_{h}(E)=d_h(E,p)$ for the distance of $E$ from $p$. Then we get
		$$
		\mathcal{H}^k(E\cap B^d_r)=f_r\circ d_{h}(E),\qquad E\in \AOp_h(d,k).
		$$
		Let $\eta$ denote a hyperbolic $k$-flat process in $\mathbb{H}^d$ with intensity measure $\mu_k$ with $\mu_k$ as introduced in Section \ref{sec:genset} . For the image measure ${d_h}_{\sharp}\mu_k$ of $\mu_k$ under $d_h$, we get
		$$
		{d_h}_{\sharp}\mu_k=\omega_{d-k}\int_0^\infty \indi\{s\in\cdot\}\,\cosh^k s\, \sinh^{d-k-1} s\, \dint s,
		$$
		see  \cite[Sections 3.4-5]{Ch93} for the required transformation in hyperbolic space. Note that since
		$$
		F_r=\int f_r\circ d_h(E)\, \eta(\dint E),
		$$
		and since the characteristic function of $F_r$ can be read off from the Laplace functional of a Poisson process and derived in essentially the same way (see, for instance,  \cite[Lemma 15.2]{Kallenberg2021}), we obtain for $\xuta\in\R$,
		\begin{align*}
			\EE[\mre^{\mri\xuta F_r}]&=\exp\left(\int_{\AOp_h(d,k)}(\mre^{\mri\xuta (f_r\circ d_h)(E)}-1)\, \mu_k(\dint E)
			\right)\\
			&=\exp\left(\int_0^r(\mre^{\mri\xuta f_r(s)}-1)\, ({d_h}_\sharp \mu_k)(\dint s)\right)\\
			&=\exp \left(\omega_{d-k} \int_0^r (\mre^{\mri\xuta f_r(s)}-1)\cosh^k s\, \sinh^{d-k-1} s\,\dint s\right).
		\end{align*}
		From this we can conclude that
		\begin{align*}
			\psi_r(\xuta):= \EE\left[\exp\left(\mri \xuta\tfrac{F_r- \EE[ F_r]}{e^{r(k-1)}}\right)\right]= \exp \left(\omega_{d-k} \int_0^r ( \mre^{\mri \xuta g_r(s)}-1 -\mri \xuta g_r(s))\cosh^k s\, \sinh^{d-k-1} s\,\dint s\right)
		\end{align*}
		with $ g_r(s):= f_r(s)/e^{r(k-1)} $. 
		
		The following lemma determines the asymptotic behaviour of $ g_r(s) $, as  $r\rightarrow \infty$, and provides a slight generalization of  \cite[Lemma 3.1]{KabluchkoRosenThaele}.
		
		\begin{lemma}
			Let $ s \in [0, \infty) $. Then $ g_r(s)\sim g(s):= \frac{\omega_k}{(k-1)2^{k-1}} \cosh^{-(k-1)} s$, as $r \rightarrow \infty$.
		\end{lemma}
		\begin{proof}
			According to \cite[Theorem 3.5.3]{Ratcliffe2019} it holds that
			\begin{align}\label{eq:gr(s)}
				g_r(s)
				= e^{-r(k-1)} \omega_k \int_0^{\arccosh( \frac{\cosh r }{\cosh s})} \sinh^{k-1} u \, \dint u.
			\end{align}
			Denoting by $ o(1) $ a sequence which converges to zero as  $z\rightarrow \infty$, we have
			\begin{align*}
				\arccosh z=\log(z+ \sqrt{z^2-1})=\log(2z)+ o(1), \quad \text{ as } z \rightarrow \infty.
			\end{align*}
			Thus, for any fixed $s\ge 0$ and as $r\to\infty$, we obtain 
			\begin{align*}
				\arccosh\Big( \frac{\cosh r}{\cosh s}\Big)= \log(2 \cosh r)- \log(\cosh s)+ o(1)=r - \log(\cosh s)+ o(1).
			\end{align*}
			Next we observe that
			\begin{align*}
				\int_{0}^{z} \sinh^{k-1} u \, \dint u 
				\sim \frac{e^{(k-1)z}}{(k-1)2^{k-1}},\quad \text{ as }z \rightarrow \infty.
			\end{align*}
			Combining this relation for $ z=r- \log(\cosh s)+ o(1) $ with \eqref{eq:gr(s)}, as $r\to\infty$, we obtain for any fixed $s\ge 0$ that 
			\begin{align*}
				g_r(s) &\sim e^{-r(k-1)}  \frac{\omega_k }{(k-1)2^{k-1}}e^{(k-1)(r - \log(\cosh s)+ o(1) )}\\
				&=  \frac{\omega_k }{(k-1)2^{k-1}} e^{(k-1)(-\log(\cosh s )+ o(1))}\\
				&\sim \frac{\omega_k }{(k-1)2^{k-1}}  \cosh^{-(k-1)} s ,
			\end{align*}
			which completes the proof of the lemma.
		\end{proof}
		
		In order to conclude that
		\begin{align}\label{eq:LimitCharacteristicFunction}
			\lim_{r \rightarrow \infty} \psi_r (\xuta)= \psi(\xuta):=\exp \left( \omega_{d-k}\int_0^\infty( e^{\mri \xuta g(s)}-1 -\mri \xuta g(s))\cosh^k s \, \sinh^{d-k-1} s \,\dint s\right)
		\end{align}
		we will use the dominated convergence theorem. To apply it, we need to find an integrable upper bound  for the   absolute value of the integrand $ (e^{\mri \xuta g_r(s)}-1 -\mri\xuta  g_r(s))\cosh^k s\, \sinh^{d-k-1}s$. Note that  for any $ s,\xuta \geq 0 $ we have  that
		\begin{align*}
			|e^{\mri\xuta g_r(s)}-1- \mri\xuta g_r(s)| \leq \frac{1}{2}\xuta^2g_r(s)^2,
		\end{align*}
		(see, e.g., \cite[Lemma 6.15]{Kallenberg2021}). 
		Using in addition that $ \cosh s \leq e^s $  and $ \sinh s \leq e^s $ for $ s \geq  0$ we  get
		\begin{align*}
			|\big(e^{\mri\xuta g_r(s)}-1 -\mri\xuta g_r(s)\big)\cosh^k s\, \sinh^{d-k-1}s| \leq \frac{1}{2}\xuta^2g_r(s)^2e^{s(d-1)}.
		\end{align*}
		Furthermore, \cite[Lemma 7]{HeroldHugThaele} provides the upper bound $g_r(s)\leq \frac{\omega_k}{k-1}e^{-s(k-1)}$, so that we obtain
		\begin{align*}
			|\big(e^{\mri\xuta g_r(s)}-1 -\mri\xuta g_r(s)\big)\cosh^k s\, \sinh^{d-k-1} s| \leq \frac{1}{2}\xuta^2 \frac{\omega_k^2}{(k-1)^2}e^{-s(2k-(d+1))}.
		\end{align*}
		In fact, the right-hand side provides an integrable function of $s\ge 0$, independent of $r> 0$, for $ 2k> d+1 $. 
		Thus,  \eqref{eq:LimitCharacteristicFunction} proves that, as  $r\rightarrow \infty$,  the random variables $ \frac{F_r- \EE F_r}{e^{r(k-1)}} $ converge in distribution to a random variable $ Y $ with characteristic function $ \psi $.  
		
		As in the introduction, for $T>0$ define the random variable 
		$$Y_T:=\sum_{s\in\zeta\cap[0,T]}h(s)$$ with $h(s):=\cosh^{-(k-1)}s$, $s\geq 0$. 
		Using once again \cite[Lemma 15.2]{Kallenberg2021}, we conclude that $Y_T$ has characteristic function 
		$$
		\psi_{Y_T}(\xuta) = \exp\Big(\omega_{d-k}\int_0^T(e^{\mri\xuta h(s)}-1)\cosh^ks\sinh^{d-k-1}s\,\dint s\Big),\qquad \xuta\in\RR.
		$$
		Moreover, $\EE Y_T = \omega_{d-k}\int_0^Th(s)\cosh^ks\sinh^{d-k-1}s\,\dint s=\frac{\omega_{d-k}}{ d-k}\sinh^{d-k}T$, which implies that the characteristic function of the centred random variable $Y_T-\EE Y_T$ is
		$$
		\psi_{Y_T-\EE Y_T}(\xi) = \exp\Big(\omega_{d-k}\int_0^T(e^{\mri\xi h(s)}-\mri \xi h(s)-1)\cosh^ks\sinh^{d-k-1}s\,\dint s\Big).
		$$
		Taking $T\to\infty$ and using the dominated convergence theorem once again, we conclude by comparing the definitions of the functions $g(s)$ and $h(s)$ that $\psi_{Y_T-\EE Y_T}(\xi)$ converges to $\psi(\xi/c_k)$ with $c_k=\frac{\omega_k }{(k-1)2^{k-1}}$. This eventually proves convergence in distribution of $Y_T-\EE Y_T$ to $Z$, where $Z$ is as in the statement of Theorem~\ref{thm:NoCLT}. The proof is thus complete.\qed
		
		\begin{remark}{\rm It follows from \eqref{eq:LimitCharacteristicFunction} that the $\ell$-th order cumulant $\cum_{\ell}(Z)$ of $Z$ equals
				\begin{align*}
					\cum_{\ell}(Z):&=(-\mri)^\ell\frac{\partial^\ell}{\partial \xuta^\ell}\log \EE\left[\exp\left(\mri \xuta Z\right)\right]\Big|_{\xi=0}\\
					&=\omega_{d-k}\int_0^\infty\cosh^{k-\ell(k-1)}s\, \sinh^{d-k-1}s\, \dint s\in (0,\infty)    
				\end{align*}
				for $\ell\in\NN$ with $\ell\ge 2$. Note that $k-\ell(k-1)+d-k-1=d+1-2k-(\ell-2)(k-1)<0$ for $2k>d+1$.  The integral can be expressed in terms of Gamma functions. For this  note that for $-a>b>-1$,
				$$
				\int_0^\infty \cosh^a s\, \sinh^b s\, \dint s=\frac{1}{2}\frac{\Gamma\left(-\frac{a+b}{2}\right)\Gamma\left(\frac{b+1}{2}\right)}{\Gamma\left(\frac{1-a}{2}\right)},
				$$
				which can be obtained by the substitution $\cosh s=v^{-1}$ which transforms the integral into the integral representation of the Beta function. 
			}
		\end{remark}
		
		\section{Proof of Theorem \ref{IGMeta}}
		
		For the proof, we will first consider the special case $k=1$, and then we derive the general result by a basic integral-geometric relation and by applying twice the special case already established.
		
		For $\kappa=0$, the case $k=1$ of Theorem \ref{IGrel} is a special case of the Euclidean affine Blaschke-Petkantschin formula \cite[Theorem 7.2.7]{SW} (with $d=k$ and $k=1$ there).
		
		For $\kappa=1$, the case $k=1$ of Theorem \eqref{IGrel} is a special case of \cite[Lemma 5.3]{HugThaeleSTIT} (with $q=2$ there) if the different normalization of the measure $\mu_{1,1}$ and the relation $\sin d_{1}(x,y)=\nabla_2(x,y)$ is taken into account, where we recall that $d_1$ stands for the geodesic distance on $\bM_1^d=\SS^d$ .
		
		For $\kappa=-1$, the case $k=1$ of Theorem \eqref{IGrel} is stated in \cite[Equation (18.2)]{Santalo} in a different language (using the classical calculus of differential forms).  In the following, we provide a different argument and additional information which should be useful for other purposes as well. More specifically, we  apply a special case of the affine Blaschke--Petkantschin formula in Euclidean space and  use the Beltrami--Klein model. We write $d_{\sfB}$ for the intrinsic distance function and $\Gammo_{\sfB}(d,k)$ for the $k$-flats in this model. If $E$ is a $k$-flat in $\R^d$ or its intersection with $\sfB^d$, then we write $\tau(E)$ for the Euclidean orthogonal projection of the origin $o\in\R^d$ to $E$. Clearly, if $k=d$, then $\tau(E)=o$.
		
		The following lemma relates the volume element of a $k$-flat $E$ in the Beltrami--Klein model to Euclidean quantities.
		Recall that $\GOp_0(d,k)$ is the linear Grassmannian of Euclidean space $\R^d$, that is, the set of all $k$-dimensional linear subspaces of $\R^d$. We write $L^\perp$ for the Euclidean orthogonal complement of a linear subspace $L\in \GOp_0(d,k)$.  The rotation invariant (Haar) probability measure on $\GOp_0(d,k)$ is denoted by $\nu_{k,0}$. The restriction of a measure $\mu$ to a $\mu$-measurable set $A$ is denoted by $\mu\llcorner A$.
		
		\begin{lemma}\label{volelk}
			Let $k\in \{1,\ldots,d\}$ and $E\in \Gammo_\sfB(d,k)$. Let  $L\in \GOp_0(d,k)$ be the unique linear subspace such that $E=(L+\tau(E))\cap \sfB^d$.  Then the restriction $\mathcal{H}^k_\sfB\llcorner E$ of the $k$-dimensional Hausdorff measure $\mathcal{H}^k_\sfB $  in the Beltrami--Klein model  to $ E$ satisfies
			$$
			\mathcal{H}^k_\sfB\llcorner E=\int \indi\{x\in\cdot\}\frac{\sqrt{1-\|\tau(E)\|^2}}{\left(1-\|x\|^2\right)^{\frac{k+1}{2 }}}\, (\cH^k_0\llcorner E )(\dint x).
			$$
		\end{lemma}
		
		\begin{proof} Recall from \cite[Theorem 6.1.5]{Ratcliffe2019} that the Riemannian metric of the Beltrami--Klein model
			at $x\in\sfB$ is given by
			\begin{equation}\label{eq:RMgB}
				g_x(u,v)=\frac{(1-\|x\|^2)u\bullet v+(x\bullet u)(x\bullet v)}{(1-\|x\|^2)^2}, \qquad u,v\in T_x\sfB,
			\end{equation}
			where the tangent space  $T_x\sfB$ of $\sfB$ at $x$ is identified with $\R^d$. Let $u_1,\ldots,u_k\in L$ be a Euclidean orthonormal basis of $L$. Let $\mathrm{I}_k$ denote the $k\times k$ identity matrix. Then the determinant $G_L(x)$ of the Gram matrix $\left(g_x(u_i,u_j)_{i,j=1}^k\right)$ is independent of the chosen orthonormal basis of $L$ and  given by
			\begin{align*}
				G_L(x)&=\frac{1}{(1-\|x\|^2)^{2k}}\det\left((1-\|x\|^2)\mathrm{I}_k+\left((x\bullet u_i)(x\bullet u_j)\right)_{i,j=1}^k\right)\\
				&=\frac{1}{(1-\|x\|^2)^{2k}}(1-\|x\|^2)^{k-1}\left(1-\|x\|^2+\sum_{i=1}^k(x\bullet u_i)^2\right)\\
				&=\frac{1-\|\tau(E)\|^2}{(1-\|x\|^2)^{k+1}}.
			\end{align*}
			For the (almost effortless) calculation of the determinant (second equality) one can use that the linear map $\psi:\R^k\to\R^k$ given by $\psi(z)=\alpha z+(a\bullet z)a$ for given $\alpha\in\R$ and $a\in\R^k$ has the eigenvalues $\alpha$ (with multiplicity $k-1$) and $\alpha+\|a\|^2$.
		\end{proof}
		
		\begin{remark}\rm 
			Note that for $k=d$ with $\tau(E)=o$ the preceding lemma relates the corresponding volume forms.
		\end{remark}
		
		Note that the non-empty intersections $E\cap \sfB^d$ with $E\in \AOp_0(d,k)$ are precisely the $k$-flats in $\Gammo_\sfB(d,k)$ (see \cite[Theorem 6.1.4]{Ratcliffe2019}). In the following, we  write $\mu_{k,0}$ for the restriction of the Euclidean isometry invariant measure on $\AOp_0(d,k)$ to the $k$-flats in $\Gammo_\sfB(d,k)$, by identifying a $k$-flat from $\AOp_0(d,k)$  and its intersection with $\sfB^d$.  The next lemma expresses the Haar measure $\mu_{k,\sfB}$ on $\Gammo_\sfB(d,k)$ in terms of Euclidean quantities, that is, we provide the density of $\mu_{k,\sfB}$ with respect to the Haar measure $\mu_{k,0}$.  In the following, we write $\nu_{d-k,0}$ for the Haar (rotation invariant) probability measure on $G_0(d,d-k)$. In the case $k=0$,  the next lemma  recovers the known relation for the corresponding volume forms.

		\begin{lemma}\label{invmeasure}
			Let $k\in\{0,1,\ldots,d-1\}$. Then
			\begin{align*}
				\mu_{k,\sfB}&=\int_{\AOp_0(d,k)}\indi \left\{E\cap \sfB^d \in\cdot\right\}\left(1-\|\tau(E)\|^2\right)^{-\frac{d+1}{2}}\, \mu_{k,0}(\dint E)\\
				&=\int_{\GOp_0(d,d-k)}\int_{L\cap \sfB^d }\indi \left\{(L^\perp+x)\cap \sfB^d\in\cdot\right\}
				\left(1-\|x\|^2\right)^{-\frac{d+1}{2}}\, \mathcal{H}^{d-k}_{0}(\dint x)\, \nu_{d-k,0}(\dint L).
			\end{align*}
		\end{lemma}
		
		\begin{proof}
			We start from the general expression for the measure $\mu_{k}$, applied in the Beltrami--Klein model. Then we express the arising distance $d_{\sfB}(o,x)$ by means of \cite[Theorem 6.1.1]{Ratcliffe2019} and use the relation for the Hausdorff measures  which is available from a special case of Lemma \ref{volelk} (see also  \cite[Theorem 6.1.6]{Ratcliffe2019}). 
			
			Let $\GOp_{\sfB}(d,d-k):= \{L\in \Gammo_\sfB(d,d-k):o\in L\}$ and write $\nu_{d-k,\sfB}$ for the isometry invariant probability measure on $\GOp_{\sfB}(d,d-k)$. Recall that for $L\in \GOp_{\sfB}(d,d-k)$ and $x\in L$, we write $H(L,x)$ for the $k$-flat through $x$ which is orthogonal to $L$ in $x$ (here orthogonality refers to the Riemannian metric $g_x$ as given at \eqref{eq:RMgB}). If $U\in \GOp_{0}(d,d-k)$, $L=U\cap \sfB^d$ and $x\in L$, then $H(L,x)=(U^\perp+x)\cap\sfB^d$, where   $U^\perp$ is the Euclidean orthogonal complement of $U$ in  $\RR^d$. To see this, it suffices to observe that $x\in (U^\perp+x)\cap\sfB^d\in \AOp_{\sfB}(d,k)$ and $U^\perp= (T_xL)^\perp$, where $(T_xL)^\perp$ means the orthogonal complement of $T_x L$ in $T_x\sfB^d$ with respect to the Riemannian metric $g_x$. In fact, if $v\in U^\perp$ and $u\in T_x L=U$, then $u\bullet v=0$ and $x\bullet v=0$ (since $x\in L\subset U$), hence $g_x(u,v)=0$, which yields $U^\perp\subset (T_xL)^\perp$ and therefore $U^\perp= (T_xL)^\perp$ (since both subspaces have the same dimension).  Thus,
			\begin{align*}
				\mu_{k,\sfB}&=\int_{\GOp_{\sfB}(d,d-k)}\int_L\indi \{H(L,x)\in \cdot\}\cosh^k d_{\sfB}(o,x)\, \mathcal{H}^{d-k}_{\sfB}(\dint x)\,\nu_{d-k,\sfB}(\dint L)\\
				&=\int_{\GOp_0(d,d-k)}\int_{U\cap \sfB^d}\indi \{(U^\perp+x)\cap \sfB^d\in \cdot\}\left(1-\|x\|^2\right)^{-\frac{k}{2}}\\
				&\qquad\qquad\qquad\qquad\qquad\qquad\times
				\left(1-\|x\|^2\right)^{-\frac{d-k+1}{2}}\, \mathcal{H}^{d-k}_{0}(\dint x)\, \nu_{d-k,0}(\dint U)\\
				&=\int_{\GOp_0(d,d-k)}\int_{U\cap \sfB^d}\indi  \{(U^\perp+x)\cap \sfB^d\in \cdot\}
				\left(1-\|x\|^2\right)^{-\frac{d+1}{2}}\, \mathcal{H}^{d-k}_{0}(\dint x)\, \nu_{d-k,0}(\dint U)\\
				&=\int_{\GOp_0(d,k)}\int_{U^\perp\cap \sfB^d}\indi \{(U+x)\cap \sfB^d\in \cdot\}
				\left(1-\|x\|^2\right)^{-\frac{d+1}{2}}\, \mathcal{H}^{d-k}_{0}(\dint x)\, \nu_{k,0}(\dint U)\\
				&=\int_{\AOp_0(d,k)}\indi\left\{E\cap \sfB^d\in\cdot\right\}\left(1-\|\tau(E)\|^2\right)^{-\frac{d+1}{2}}\, \mu_{k,0}(\dint E),
			\end{align*}
			which yields the assertion.
		\end{proof}

		Finally, we prepare the proof of Theorem  \ref{IGMeta} by providing another basic relation. Note that here it is crucial that $E\in\Gammo_\sfB(d,k)$ with $k=1$. Hence, for $x,y\in E\in\Gammo_\sfB(d,1)$, \cite[Theorem 6.1.1]{Ratcliffe2019} yields
		\begin{align}
			\sinh d_\sfB(x,y)&=\sqrt{\cosh^2 d_\sfB(x,y)-1}\nonumber\\
			&=\frac{\sqrt{(1-x\bullet y)^2-(1-\|x\|^2)(1-\|y\|^2)}}{\sqrt{1-\|x\|^2}\sqrt{1-\|y\|^2}}\nonumber\\
			&=\frac{\sqrt{ \|x-y\|^2+(x\bullet y)^2-\|x\|^2\|y\|^2}}{\sqrt{1-\|x\|^2}\sqrt{1-\|y\|^2}}\nonumber\\
			&=\frac{\|x-y\|\sqrt{1-\|\tau({E})\|^2}}{\sqrt{1-\|x\|^2}\sqrt{1-\|y\|^2}},\label{eqdh3}
		\end{align}
		where we used the fact that $\|x\|^2\|y\|^2-(x\bullet y)^2=\|x-y\|^2 \|\tau({E})\|^2$.

		\medskip 
		
		Now we are prepared for the proof of the following special case of Theorem \ref{IGMeta}.
		
		\begin{lemma}\label{specialh}
			Theorem \ref{IGMeta} holds for $\kappa=-1$ and $k=1$.
		\end{lemma}
		
		\begin{proof}
			First, we express the integration in
			$$
			I:=\int_{\HH^d} \int_{\HH^d} f(x,y)\, \cH^d(\intd x)\, \cH^d(\intd y)
			$$
			in the Beltrami--Klein model
			by an integration with respect to  Euclidean Hausdorff measures and densities given by the volume form in \cite[Theorem 6.1.6]{Ratcliffe2019}. In other words, there is an isometry $\varphi:\sfB^d \to \HH^d$ such that
			$$
			I=\int_{\sfB^d }\int_{\sfB^d } f(\varphi(x),\varphi(y))\frac{1}{(1-\|x\|^2)^{\frac{d+1}{2}}}\frac{1}{(1-\|y\|^2)^{\frac{d+1}{2}}}\, \cH^d_0(\intd x)\, \cH^d_0(\intd y).
			$$
			From the Euclidean affine Blaschke--Petkantschin formula (with $k=1$, see above), we deduce that
			$$
			I=\frac{\omega_d}{2}\int_{\AOp_0(d,1)}\int_{\sfB^d\cap E}\int_{\sfB^d\cap E} f(\varphi(x),\varphi(y)) \frac{\|x-y\|^{d-1}}{(1-\|x\|^2)^{\frac{d+1}{2}}(1-\|y\|^2)^{\frac{d+1}{2}}}
			\,\cH^1_0(\intd x)\,\cH^1_0(\intd y)\, \mu_{1,0}(\intd E).
			$$
			This can be rewritten in the form
			\begin{align*}
				I&=\frac{\omega_d}{2}\int_{\AOp_0(d,1)}\int_{\sfB^d\cap E} \int_{\sfB^d\cap E} f(\varphi(x),\varphi(y))\frac{\|x-y\|^{d-1}\left(1-\|\tau(E)\|^2\right)^{\frac{d-1}{2}}}{(1-\|x\|^2)^{\frac{d-1}{2}}
					(1-\|y\|^2)^{\frac{d-1}{2}}} \\
				&\qquad\qquad \times
				\frac{\sqrt{1-\|\tau(E)\|^2}}{1-\|x\|^2}\frac{\sqrt{1-\|\tau(E)\|^2}}{1-\|y\|^2}\left(1-\|\tau(E)\|^2\right)^{-\frac{d+1}{2}}
				\,\cH^1_0(\intd x)\,\cH^1_0(\intd y)\, \mu_{1,0}(\intd E).
			\end{align*}
			Next we apply \eqref{eqdh3}, Lemma \ref{volelk} (twice) and Lemma \ref{invmeasure} with $k=1$ to get
			\begin{align*}
				I&=\frac{\omega_d}{2}\int_{\Gammo_\sfB(d,1)}\int_{\sfB^d\cap E} \int_{\sfB^d\cap E} f(\varphi(x),\varphi(y))\sinh^{d-1}d_\sfB(x,y)
				\,\cH^1_{\sfB}(\intd x)\,\cH^1_{\sfB}(\intd y)\, \mu_{1,\sfB}(\intd {E}).
			\end{align*}
			Since $\varphi:\sfB^d\to \bM_{-1}^d$ is an isometry, we finally obtain
			$$
			I=\frac{\omega_d}{2}\int_{\AOp_{h}(d,1)}\int_{\HH^d \cap  G} \int_{\HH^d \cap  G} f(x,y) \sinh^{d-1} d_{h}(x,y)
			\,\cH^1(\intd x)\,\cH^1(\intd y)\, \mu_{1}(\intd G),
			$$
			as asserted.
		\end{proof}

		In order to deduce Theorem \ref{IGMeta} for general $k$ from the case where $k=1$, the following simple lemma will be useful. For $E\in \AOp_\kappa(d,k)$  we write $\AOp_\kappa(E,1)$ for the $1$-flats of $\bM_\kappa^d$ lying in $E$ (which then are also  $1$-flats of $E$), where $E$ is considered as a $k$-dimensional hyperbolic space. In this case, we write $\mu^{E}_{1,\kappa}$ for the consistently normalized
		invariant measure on $\AOp_\kappa(E,1)$, where the invariance refers to isometries of $E$. In particular, the normalization is independent of $E$, that is
		$$
		\mu^{E}_{1,\kappa}(\{G\in \AOp_\kappa(E,1):G\cap B_\kappa^E(1)\neq \emptyset\})
		$$
		is independent of $E$ and the choice of a geodesic ball $B_\kappa^E(1)$ of radius $1$ in $E$.
		The following auxiliary result is stated in \cite[Equation (2.5)]{Solanes} in the language of differential forms, the Euclidean and the spherical case are established in \cite[Section 7.1]{SW}. We argue in a different way.
		
		\begin{lemma}\label{Lemmadh2}
			Let $g:\AOp_\kappa(d,1)\to [0,\infty]$ and $k\in \{2,\ldots,d-1\}$. Then
			$$
			\int_{\AOp_\kappa(d,k)}\int_{\AOp_\kappa(E,1)} g(G)\, \mu^{E}_{1,\kappa}(\intd G)\, \mu_{k,\kappa}(\intd E)=
			\int_{\AOp_\kappa(d,1)} g(G)\, \mu_{1,\kappa}(\intd G).
			$$
		\end{lemma}
		
		\begin{proof}
			Both sides define isometry invariant measures on $\AOp_\kappa(d,1)$ (if $g$ is chosen as the indicator function of a measurable set).
			For the right side, this is clear by construction. To see this also for the left side, we argue as follows.
			Let $E\in \AOp_\kappa(d,k)$ be fixed (for the moment) and let $\varphi$ be an isometry of $\bM_\kappa^d$. For a measurable set $A\subset \AOp_\kappa( \varphi (E),1)$,  we define
			$$
			\mu(A):=\int_{\AOp_\kappa(E,1)}\indi\{\varphi (G)\in A\}\, \mu_{1,\kappa}^{E}(\intd G).
			$$
			Since the isometry $\varphi$ maps $1$-flats of $E$ bijectively to $1$-flats of $\varphi(E)$, it is easy to see that $\mu$ is a locally finite measure on
			the Borel sets of $\AOp_\kappa( \varphi (E),1)$. Let $\iota:\varphi (E)\to \varphi (E)$ be an isometry of $\varphi (E)$. Then $\varphi^{-1}\circ \iota^{-1}\circ \varphi:E\to E$ is an isometry of $E$. Moreover,
			\begin{align*}
				\mu(\iota (A))&=\int_{\AOp_\kappa(E,1)}\indi\{\varphi (G)\in \iota (A)\}\, \mu^E_{1,\kappa}(\intd G)\\
				&=\int_{\AOp_\kappa(E,1)}\indi\{\varphi  ((\varphi^{-1}\circ \iota^{-1}\circ \varphi) (G))\in A\}\,\mu^E_{1,\kappa}(\intd G)\\
				&=\int_{\AOp_\kappa(E,1)}\indi\{\varphi  (\overline{G}) \in A\}\, \mu^{E}_{1,\kappa}(\intd \overline{G})\\
				&=\mu(A),
			\end{align*}
			where in the second to last step we used that $\mu^{E}_{1,\kappa}$ is invariant with respect to isometries of $E$. If $B_\kappa^E(1)$ is a geodesic ball of radius $1$ in $E$, then $\varphi( B_\kappa^E(1))$ is a geodesic ball of radius $1$ in $\varphi (E)$ and
			$$
			\mu(\{G\in \AOp_\kappa(E,1):G\cap \varphi (B_\kappa^E(1))\neq\emptyset\})=\int_{\AOp_\kappa(E,1)}\indi\{G\cap B_\kappa^E(1)\neq\emptyset\}\,
			\mu_{1,\kappa}^{E}(\intd G).
			$$
			Hence $\mu$ is the consistently normalized Haar measure on $\AOp_\kappa(\varphi (E),1)$. In particular, for each $E\in \AOp_\kappa(d,k)$ and each isometry $\varphi$ of $\bM^d_\kappa$, we have
			$$
			\int_{\AOp_\kappa(E,1)}g(\varphi (G))\, \mu_{1,\kappa}^{E}(\intd G)
			=\int_{\AOp_\kappa(\varphi E,1)} g(\overline{G})\, \mu_{1,\kappa}^{\varphi (E)}(\intd \overline{G})
			$$
			for each measurable function $g:\AOp_\kappa(E,1)\to[0,\infty]$.
			
			Now let $\varphi$ be an isometry of $\bM_\kappa^d$. Using first the isometry invariance of $\mu_{k,\kappa}$ and then the preceding
			considerations, we obtain
			\begin{align*}
				&\int_{\AOp_\kappa(d,k)}\int_{\AOp_\kappa(E,1)} g(\varphi (G))\, \mu^{E}_{1,\kappa}(\intd G)\, \mu_{k,\kappa}(\intd E)\\
				&\qquad =\int_{\AOp_\kappa(d,k)}\int_{\AOp_\kappa(\varphi^{-1} (E),1)} g(\varphi (G))\, \mu^{\varphi^{-1} (E)}_{1,\kappa}(\intd G)\, \mu_{k,\kappa}(\intd E)\\
				&\qquad =\int_{\AOp_\kappa(d,k)}\int_{\AOp_\kappa(  E,1)} g(\varphi (\varphi^{-1} (G)))\, \mu^{  E }_{1,\kappa}(\intd G)\,\mu_{k,\kappa}(\intd E)\\
				&\qquad =\int_{\AOp_\kappa(d,k)}\int_{\AOp_\kappa(E,1)} g( G)\, \mu^{E }_{1,\kappa}(\intd G)\, \mu_{k,\kappa}(\intd E),
			\end{align*}
			which confirms the isometry invariance of the left side.

			Since up to a scalar multiple there is only one Haar measure on $\AOp_\kappa(d,1)$, it remains to show that this multiple is $1$.
			For this, we choose $g(G):=\mathcal{H}^1_\kappa(G\cap B_{1,\kappa}^d)$, where $B_{1,\kappa}^d$ is a geodesic ball of radius $1$ centred at an arbitrary point of $\bM_\kappa^d$.
			Then the right side is $\mathcal{H}^d_\kappa(B_{1,\kappa}^d)$ by a straightforward application of the Crofton formula  \eqref{eq:CroftonFormula}.
			
			If we first apply the Crofton formula  \eqref{eq:CroftonFormula} within $E$ to the inner integral on the left side, we get
			\begin{align*}
				\int_{\AOp_\kappa(E,1)} \cH^1_\kappa(G\cap B_{1,\kappa}^d)\, \mu^{E}_{1,\kappa}(\intd G)&=
				\int_{\AOp_\kappa(E,1)} \cH^1_\kappa(G\cap (E\cap B_{1,\kappa}^d))\, \mu^{E}_{1,\kappa}(\intd G)\\
				&=\mathcal{H}^k_\kappa(E\cap B_{1,\kappa}^d).
			\end{align*}
			Then another application of the Crofton formula \eqref{eq:CroftonFormula} also yields $\mathcal{H}^d_\kappa(B_{1,\kappa}^d)$ for the double integral on the left-hand side.
		\end{proof}
		
		\begin{remark}\rm 
			By the same reasoning, we also get a corresponding result for  integrals over flags $(E,G)\in \AOp_\kappa(d,k)\times \AOp_\kappa(d,p)$  with $G\subset E$ (and fixed $p<k$).
		\end{remark}
		
		\begin{proof}[Proof of Theorem \ref{IGMeta}] The assertion in the case $k=1$ has already been established.
			Let $k\in \{2,\ldots,d-1\}$. Then by the case $k=1$ and by Lemma \ref{Lemmadh2}, we get
			\begin{align}
				&\int_{\bM_\kappa^d}\int_{\bM_\kappa^d}f(x,y)\,\cH^d_\kappa(\dint x)\,\cH^d_\kappa(\dint y)\nonumber\\
				&=\frac{\omega_d}{2}\int_{\AOp_\kappa(d,1)}\int_{\bM_\kappa^d \cap  G} \int_{\bM_\kappa^d \cap  G} f(x,y) \sn_\kappa^{d-k} d_\kappa(x,y)\sn_\kappa^{k-1} d_\kappa(x,y)
				\,\cH^1_{\kappa}(\intd x)\,\cH^1_{\kappa}(\intd y)\, \mu_{1,\kappa}(\intd G)\nonumber\\
				&=\frac{\omega_d}{2}\int_{\AOp_\kappa(d,k)}\int_{\AOp_\kappa(E,1)}\int_{\bM^d_\kappa \cap  G} \int_{\bM^d_\kappa \cap  G} g(x,y) \nonumber\\
				&\hspace{4cm}\times \sn_\kappa^{k-1} d_\kappa(x,y)
				\,\cH^1_{\kappa}(\intd x)\,\cH^1_{\kappa}(\intd y)\, \mu^{E}_{1,\kappa}(\intd G)\, \mu_{k,\kappa}(\intd E),\label{dhsubst}
			\end{align}
			where $g(x,y):=f(x,y)\sn_\kappa^{d-k}d_\kappa(x,y)$. Since $\bM^d_\kappa\cap E$ is a $k$-flat and a $k$-dimensional space of constant curvature $\bM^E_\kappa$,
			we can apply the result that has already been proved to the integrand of the integration over $\AOp_\kappa(d,k)$ for each fixed $E\in \AOp_\kappa(d,k)$. This implies that
			\begin{align}\label{eq:subst}
				&\frac{\omega_k}{2}\int_{\AOp_\kappa(E,1)}\int_{\bM^d_k \cap  G} \int_{\bM^d_k \cap  G} g(x,y)  \sn_\kappa^{k-1} d_\kappa(x,y)
				\,\cH^1_{\kappa}(\intd x)\,\cH^1_{\kappa}(\intd y)\, \mu^{E}_{1,\kappa}(\intd G)\nonumber\\
				&\qquad = \int_{\bM^d_\kappa\cap E}\int_{\bM^d_\kappa\cap E}g(x,y) \, \cH^k_\kappa(\intd x)\, \cH^k_\kappa(\intd y).
			\end{align}
			Substituting \eqref{eq:subst} into \eqref{dhsubst}, we get the asserted equation.
		\end{proof}

		\begin{corollary}\label{corIG}
			Let $A\subset \bM^d_\kappa $ be a measurable set. If $k\in \{1,\ldots,d-1\}$, then
			$$
			\int_{\AOp_\kappa(d,k)} \cH^k_\kappa(A\cap E)^2\, \mu_{k,\kappa}(\intd E)
			= \frac{\omega_k}{\omega_d}\int_{A} \int_{A} \sn_\kappa^{-(d-k)}d_\kappa(x,y)\, \cH^d_\kappa(\intd x)\, \cH^d_\kappa(\intd y).
			$$
		\end{corollary}
		
		\begin{proof} We apply Theorem \ref{IGMeta} with 
			$$f(x,y):=\indi\{x\neq y\}\indi_A(x)\indi_A(y)\sn_\kappa^{-(d-k)} d_\kappa(x,y),\qquad x,y\in \bM^d_\kappa,$$
			where by definition $f(x,y)=0$ if $x=y$. Thus we get
			\begin{align}\label{eq:bothsides}
				&      \int_{\AOp_\kappa(d,k)}\int_{E\cap A}\int_{E\cap A}\indi\{x\neq y\} \,\cH_\kappa^k(\dint x)\,\cH_\kappa^k(\dint y)\,\mu_{k,\kappa}(\dint E)\nonumber\\
				&\qquad = \frac{\omega_k}{\omega_d}\int_{A}\int_{A} \indi\{x\neq y\} \sn_\kappa^{-(d-k)}d_\kappa(x,y)\,\cH^d_\kappa(\dint x)\,\cH^d_\kappa(\dint y),
			\end{align}
			where $\indi\{x\neq y\} \sn_\kappa^{-(d-k)}d_\kappa(x,y)=0$ if $x=y$ (by definition). On the left side of this equation, for $\mu_{k,\kappa}$-almost every $E\in \AOp_\kappa(d,k)$ such that $E\cap A\neq \emptyset$, the indicator function is not equal to $1$ only on a set of $(x,y)\in (E\cap A)^2$ of $(\cH_\kappa^k)^2$-measure zero. Similarly,  the indicator function on the right side is not equal to $1$ only on a set of $(x,y)\in A^2$ of $(\cH^d_\kappa)^2$-measure zero. Since the integral of a function with values in $[0,\infty]$ over a set of measure zero is zero, we can omit the indicator functions on both sides of \eqref{eq:bothsides},  and the asserted equation follows.
		\end{proof}
		
		\section{Proof of Theorem \ref{covmax}}
		
		The proof of Theorem \ref{covmax} relies on the sharp Riesz rearrangement inequality \cite[Theorem 2]{BS01}. To keep our paper self-contained, we rephrase here a special case of this inequality in a slightly more general setting to which the statement and  proof of \cite[Theorem 2]{BS01} (see also \cite[Theorem 2]{BuHa07}) can be extended. For an integrable function $f:\bM_\kappa^d\to[0,\infty)$ we denote by $f^*$ the symmetric decreasing rearrangement of $f$ with respect to the fixed origin $p$ of $\bM_\kappa^d$. (By decreasing we mean non-increasing.) To recall the definition of $f^*$, we write $\{f>s\}:=f^{-1}((s,\infty))$ for $s>0$ and denote by $\{f>s\}^*$ the open geodesic ball  with center at $p$ such that $\mathcal{H}^d_\kappa(\{f>s\}^*)=\mathcal{H}^d_\kappa(\{f>s\})<\infty$. Then $f^*:\bM_\kappa^d\to[0,\infty)$, defined by 
		$$
		f^*(x):=\int_0^\infty\indi_{\{f>s\}^*} (x)\, \dint s,\qquad x\in\bM_\kappa^d, 
		$$
		is a decreasing function of the geodesic distance to $p$, invariant under isometries fixing $p$,  lower semicontinuous, and $\{f^*>t\}=\{f>t\}^*$, that is, $f$ and $f^*$ are equidistributed (equimeasurable) in the sense that 
		$$
		\mathcal{H}^d_\kappa(\{f>t\})=\mathcal{H}^d_\kappa(\{f^*>t\})
		$$
		for $t>0$. In the following, we say that $f$ is non-zero if $\cH_\kappa^d(\{f>0\})>0$.
		
		\begin{lemma}[Sharp Riesz rearrangement inequality]\label{lem:Riesz}
			Let $f,g:\bM_\kappa^d\to[0,\infty)$ be $\cH_\kappa^d$-integrable functions and $K:[0,\infty)\to[0,\infty]$ be a decreasing function. Define
			$$
			\mathcal{I}_K(f,g) := \int_{\bM_\kappa^d}\int_{\bM_\kappa^d}f(x)g(y)K(d_\kappa(x,y))\,\cH_\kappa^d(\dint x)\,\cH_\kappa^d(\dint y).
			$$
			Then
			\begin{equation}\label{eq:RiszInequality}
				\mathcal{I}_K(f,g) \leq \mathcal{I}_K(f^*,g^*).
			\end{equation}
			Moreover, if $K$ is  strictly decreasing, $f$ and $g$ are non-zero,  and $\mathcal{I}_K(f^*,g^*)<\infty$, then equality in \eqref{eq:RiszInequality} holds if and only if there is some isometry $\varphi\in I(\bM_\kappa^d)$ such that $f=f^*\circ\varphi$ and $g=g^*\circ\varphi$ $\cH_\kappa^d$-almost everywhere.
		\end{lemma}
		
		We can now proceed to the proof of Theorem \ref{covmax}.
		
		\begin{proof}[Proof of Theorem \ref{covmax}]
			From \eqref{eq:Variance} we have that
			$$
			\var F_{W,t,\kappa}^{(m)}= \sum_{i=1}^m i!\, t^{2m-i}\, A_{W,\kappa,i}^{(m)}
			$$
			with $A_{W,\kappa,i}^{(m)}$ given by \eqref{eq:DefArim}. That is,
			\begin{align}\label{eqcor}
				\var F_{W,t,\kappa}^{(m)}= \sum_{i=1}^m c_{i} t^{2m-i}\int_{\AOp_\kappa(d,d-i(d-k)) }\ \cH_\kappa^{d-i(d-k)}(E \cap W )^2 \ \mu_{d-i(d-k),\kappa}(\dint E)
			\end{align}
			with the constants $c_i$, $i\in\{1,\ldots,m\}$, given by 
			$$
			c_i :=  \binom{m }{i} ^2 \frac{i!}{(m!)^2}\Big(\frac{\omega_{d+1}}{\omega_{k+1}}\Big)^{2m-i}\frac{\omega_{d-m(d-k)+1}^2 }{\omega_{d+1}\omega_{d-i(d-k)+1}}.
			$$ 
			Now, Corollary \ref{corIG} can be applied to each of the integrals on the right side of \eqref{eqcor} for which $i(d-k)\le d-1$. Since $k\ge 1$,  this condition is satisfied at least for $i=1$. Moreover, if $m(d-k)=d$, then the corresponding summand is $c_mt^m\mathcal{H}^d_\kappa(W)=c_mt^m\mathcal{H}^d_\kappa(B_W)$.  For $i(d-k)\le d-1$ we get
			\begin{align*}
				&\int_{\AOp_\kappa(d,d-i(d-k)) }\ \cH_\kappa^{d-i(d-k)}(E \cap W )^2 \ \mu_{d-i(d-k),\kappa}(\dint E)\\
				&\qquad= \frac{\omega_{d-i(d-k)}}{\omega_d}\int_{W} \int_{W} \frac{1}{\sn^{i(d-k)}d_\kappa(x,y)}\, \cH^d_\kappa(\intd x)\, \cH^d_\kappa(\intd y).
			\end{align*}
			Our goal is  to apply the sharp Riesz rearrangement inequality in Lemma \ref{lem:Riesz} with $f=g={\bf 1}_W$ and $K(s)=\sn^{-i(d-k)} s$ for $s\ge 0$.  For this, we first note that $f$ and $g$ are both $\cH_\kappa^d$-integrable and non-zero, since $W$ is a Borel set with $\cH_\kappa^d(W)\in(0,\infty)$, and that the function $r\mapsto\sn^{-i(d-k)}r$ is strictly decreasing on $[0,\infty)$ if $\kappa\in\{-1,0\}$ and on $[0,\pi/2]$ if $\kappa=1$. 
			Thus, Lemma \ref{lem:Riesz}  together with the additional assumption on $W$ in the case $\kappa=1$ and the fact that $f^*=g^*={\bf 1}_{B_W}$ show  that for $i(d-k)\le d-1$ we have
			\begin{align}
				&\frac{\omega_{d-i(d-k)}}{\omega_d}\int_{W} \int_{W} \frac{1}{\sn^{i(d-k)}d_\kappa(x,y)}\, \cH^d_\kappa(\intd x)\, \cH^d_\kappa(\intd y)\nonumber\\
				&\qquad\le \frac{\omega_{d-i(d-k)}}{\omega_d}\int_{B_W} \int_{B_W} \frac{1}{\sn^{i(d-k)}d_\kappa(x,y)}\, \cH^d_\kappa(\intd x)\, \cH^d_\kappa(\intd y)\label{eq:geomext}\\
				&\qquad=\int_{\AOp_\kappa(d,d-i(d-k)) }\ \cH_\kappa^{d-i(d-k)}(E \cap B_W )^2 \ \mu_{d-i(d-k),\kappa}(\dint E),\nonumber
			\end{align}
			where we applied backwards Corollary \ref{corIG} in the last step. Altogether we arrive at
			\begin{align*}
				\var F_{W,t,\kappa}^{(m)} &= \sum_{i=1}^m c_{i} t^{2m-i}\int_{\AOp_\kappa(d,d-i(d-k)) }\ \cH_\kappa^{d-i(d-k)}(E \cap W )^2 \ \mu_{d-i(d-k),\kappa}(\dint E)\\
				&\le \sum_{i=1}^m c_{i} t^{2m-i}\int_{\AOp_\kappa(d,d-i(d-k)) }\ \cH_\kappa^{d-i(d-k)}(E \cap B_W )^2 \ \mu_{d-i(d-k),\kappa}(\dint E)\\
				&=\var F_{B_W,t,\kappa}^{(m)}.
			\end{align*}
			The discussion of the equality case also follows from Lemma \ref{lem:Riesz}. This completes the proof of Theorem \ref{covmax}.
		\end{proof}
		
		\begin{remark}\rm
			\begin{itemize}
				\item[(i)] For compact sets $W$ and in Euclidean space, a result by Pfiefer \cite[Theorems 1 and 2]{Pfiefer90} (stated in \cite[Theorem 8.6.5]{SW} for convex bodies only) could alternatively be used at \eqref{eq:geomext}.
				\item[(ii)] Morpurgo \cite[Theorems 3.5 and 3.6]{Morpurgo02} provides a detailed proof of general Riesz rearrangement inequalities in Euclidean space, similar as in \cite{BS01}, but an investigation of  the hyperbolic case is not included. The arguments extend to the spherical case, as pointed out in \cite[page 518]{Morpurgo02}. 
				\item[(iii)] In the Euclidean setting, the inequality needed at \eqref{eq:geomext} is also provided in \cite[Theorems 3.7 and 3.9]{LL01}, \cite[Lemma 3]{Lieb77} or \cite[Corollary 2.19]{Baernstein19}, the spherical case is covered by Corollary 7.1 and Theorem 7.3 in \cite{Baernstein19}. In the special case where $f,g$ are indicator functions (as in our application) and for $\kappa=0$ (Euclidean space), the required inequality and the equality discussion can be derived from \cite[Theorem 1]{Bu96}.
			\end{itemize}
		\end{remark}

		\subsection*{Acknowledgement}
		The authors were supported by the DFG priority program SPP 2265 \textit{Random Geometric Systems}. CT was also supported by the German Research Foundation (DFG) via CRC/TRR 191 \textit{Symplectic Structures in Geometry, Algebra and Dynamics}. The authors are grateful to Almut Burchard for helpful comments on Riesz rearrangement inequalities.
		
		\addcontentsline{toc}{section}{References}
		

	\end{document}